\documentclass[oneside,english,reqno]{amsart}
\usepackage[T1]{fontenc}
\usepackage[latin9]{inputenc}
\usepackage{amsthm}
\usepackage{amssymb}
\usepackage{amsmath}
\usepackage{xcolor}
\usepackage[normalem]{ulem}
\usepackage{cancel}
\usepackage{esint}
\usepackage[margin=0.8in]{geometry}
\makeatletter

\usepackage{amsthm,amsmath,amsfonts,amssymb,color}
\usepackage[bookmarks]{hyperref}
\usepackage{mathtools}

\numberwithin{equation}{section}
\numberwithin{figure}{section}
\theoremstyle{plain}
\newtheorem{thm}{Theorem}[section]
\newtheorem{prop}[thm]{Proposition}

\newtheorem{lem}[thm]{Lemma}
\newtheorem{cor}[thm]{Corollary}
\newtheorem{claim}[thm]{Claim}
\newtheorem{rem}[thm]{Remark}

  \newcounter{casectr}
  
\theoremstyle{definition}

\theoremstyle{remark}

\newcommand{\RR}{\mathbb{R}}
\newcommand{\ZZ}{\mathbb{Z}}

\newcommand{\RRR}{\mathbb{R}}

\newcommand{\vv}{\upsilon}

\makeatother

\usepackage{fancyhdr}
\usepackage{babel}
\providecommand{\casename}{Case}

\begin{document}
\title{On a bilinear Strichartz estimate on irrational tori}

\author{Chenjie Fan $^\ddagger$ $^*$}
\author{Gigliola Staffilani $^\dagger$ $^{*}$}
\author{Hong Wang $^\S$}
\author{Bobby Wilson $^\sharp$}
\thanks{
\quad \\$^{*}$ Chenjie Fan and Gigliola Staffilani are partially supported by NSF    DMS 1362509 and DMS 1462401.\\
\quad \\$^{\ddagger} $Department of Mathematics,  University of Chicago,  5734 S University Ave, 
Chicago, IL 60637,  USA. email:
cjfan@math.uchicago.edu.\\$^{\dagger}$ Department of Mathematics,  Massachusetts Institute of Technology,  77 Massachusetts Ave,  Cambridge,  MA 02139-4307 USA. email:
gigliola@math.mit.edu.\\$^{\S}$Department of Mathematics,  Massachusetts Institute of Technology,  77 Massachusetts Ave,  Cambridge,  MA 02139-4307 USA. email:
hongwang@mit.edu.\\ $^{\sharp}$Department of Mathematics,  Massachusetts Institute of Technology,  77 Massachusetts Ave,  Cambridge,  MA 02139-4307 USA. email:
blwilson@mit.edu.}
\maketitle
\begin{abstract}
We prove a bilinear Strichartz type estimate for irrational tori via a decoupling type argument, \cite{bourgain2014proof}, recovering and generalizing the result of \cite{de2006global}. As a corollary, we derive a global well-posedness result for the  cubic defocusing NLS on two dimensional irrational tori with data of infinite energy. 
\end{abstract}
\section{Introduction}
In  \cite{bourgain2014proof} Bourgain and Demeter proved the full range of Strichartz estimates for the Schr\"odingier equation on tori as a consequence of the $L^2$ decoupling theorem. In this paper we prove in full generality the analog of the improved Strichartz estimate that first appeared in \cite{de2006global} for rational tori.

\subsection{Statement of the problem and main results}

Let $\mathbb{T}=\RR/\ZZ$ be the one dimensional torus, and let  $\alpha_{1},..,\alpha_{d-1} \in [1/2, 1]$,  we define $d$--dimensional torus $\mathbb{T}^d$ as $\mathbb{T}^d = \mathbb{T} \times \alpha_1 \mathbb{T}\times \cdots \times \alpha_{d-1} \mathbb{T}$. We say that the torus is irrational if at least on $\alpha_{i}$ is irrational.  The torus is rational otherwise.   For any $\lambda \geq 1$, we define $\mathbb{T}^d_{\lambda}$ as a rescaling of $\mathbb{T}^d$ by $\lambda$, i.e. $\mathbb{T}^d_{\lambda}=\lambda\mathbb{T}^d=(\lambda \mathbb{T})\times (\alpha_{1}\lambda\mathbb{T})...\times (\alpha_{d-1}\lambda\mathbb{T})$.

When $\lambda \rightarrow \infty$, one should think $\mathbb{T}_{\lambda}$ as a large torus approximating $\RRR^d$. 
We consider the following linear Schr\"odinger equation on $\mathbb{T}_{\lambda}$,  We consider the following Cauchy problem for the linear Schr\"odinger equation on $\mathbb{T}^d_{\lambda}$, 
\begin{equation}\label{lambda}
\begin{cases}
iu_{t}-\Delta u=0, ~ (t,x) \in \RRR\times\mathbb{T}^d_{\lambda} ;\\
u(0,x)=u_{0}, ~  u_{0} \in L^{2}( \mathbb{T}^d_{\lambda} ).
\end{cases}
\end{equation}

Let $U_{\lambda}(t)u_{0}$ be the solution to \eqref{lambda}, and let  $\Lambda_{\lambda}:=\frac{1}{\lambda}(\ZZ \times\frac{1}{\alpha_1} \ZZ\times\cdots\times \frac{1}{\alpha_{d-1}}\ZZ)$.  One has
\begin{equation}
U_{\lambda}(t)u_{0}(x)=\frac{1}{\lambda^{d/2}}\sum_{k\in\Lambda_{\lambda}} e^{2\pi k i x - |2\pi k|^2 it } \widehat{u}_0(k).
\end{equation}

Our main theorem is the following bi-linear refined Strichartz estimate.
\begin{thm}\label{Main Theorem}

Let $\phi_1, \phi_2 \in L^2(\mathbb{T}_{\lambda})$ be two initial data such that $supp \,  \hat{\phi}_i \subset \{k: |k|\sim  N_i\}, i=1,2$, for some large $N_1 \geq N_2$, and let $\eta(t)$ be a time cut-off function, $supp \, \eta \subset [0,1]$. Then 

when $d=2$, 
\begin{equation}\label{main estimate}
\|\eta(t)U_{\lambda}\phi_1 \cdot \eta(t) U_{\lambda}\phi_2\|_{L^2_{x,t}} \lesssim N_2^{\epsilon} \left(\frac{1}{\lambda}+\frac{N_2}{N_1}\right)^{1/2} \|\phi_1\|_{L^2}
 \|\phi_2\|_{L^2},
\end{equation}

when $d\geq 3$
\begin{equation}\label{mainestimatehigh}
\|\eta(t)U_{\lambda}\phi_1 \cdot \eta(t) U_{\lambda}\phi_2\|_{L^2_{x,t}} \lesssim N_2^{\epsilon} \left(\frac{N_{2}^{d-3}}{\lambda}+\frac{N_2^{d-1}}{N_1}\right)^{1/2} \|\phi_1\|_{L^2} \|\phi_2\|_{L^2}.
\end{equation}
\end{thm}

We note that
when $d=2, N_1 =N_2$, $\lambda=1$, estimate \eqref{main estimate} recovers the Strichartz inequality for the (irrational) torus after an application of H\"older's inequality, up to an $N_2^\epsilon$--loss.  When $\lambda \rightarrow \infty$, estimate \eqref{main estimate}, \eqref{mainestimatehigh}   consistent with the billinear Strichartz inequality in $\RR^{d+1}$, \cite{bourgain1998refinements}. Up to the $N_2^\epsilon$--loss,  inequality \eqref{main estimate} is sharp. 

Furthermore,
when $\lambda \geq N_1$, the estimates fall into the so-called semiclassical regime in which the geometry of $\mathbb{T}_{\lambda}$ is irrelevant.  We refer to the work of Hani, \cite{hani2012bilinear}, for same estimate (without $N_2^\epsilon$ loss) on general compact manifolds. On the torus, our results improves the estimate in \cite{hani2012bilinear} for $\lambda\leq N_{1}$. Estimate \eqref{main estimate}, \eqref{mainestimatehigh} rely on the geometry of torus and cannot hold on general compact manifolds.

\begin{rem}
It may also be interesting to consider trilinear estimates.  In fact  when one considers the quintic 	nonlinear Schr\"odinger equation as in \cite{herr2011global} and \cite{ionescu2012energy}, trilinear estimates are fundamental. See also \cite{ramos2016trilinear}.
\end{rem}

We will  derive  Theorem \ref{Main Theorem} from some bilinear decoupling type estimates. We first introduce some basic notations.

Let $P$ be the  truncated paraboloid  in $\RR^{d+1}$, 
\begin{equation}
P =\{ (\xi, |\xi|^2): \xi \in \RR^d, |\xi|\lesssim 1\}.
\end{equation}
For any function $f$ supported on $P$, we define 
\begin{equation}\label{def1}
Ef = \widehat{f d\sigma},
\end{equation}
where $\sigma$ is the measure on $P$.

Note a function supported on $P$ can be naturally understood as a function supported on the ball $B=\{\xi\in \RRR^{d}, |\xi|\lesssim 1\}$.

 By a slight abuse of notation,  for a function $f$ supported in the  ball $B$ in $\RR^d$, we also define
 \begin{equation}\label{def2}
  Ef(x,t) = \int_{B} e^{-2\pi i (\xi \cdot x +|\xi|^2 t)}f(\xi) d\xi.
\end{equation}  
One can see that the two definitions of $Ef$ are essentially the same since $P$ projects onto $B$.

We decompose $P$ as a finitely overlapping union of caps $\theta$ of radius $\delta$. Here a cap $\theta$ of radius $\delta$ is  the set $\theta=\{\xi\in P, |\xi-\xi_{0}|\lesssim \delta\}$ for some fixed $\xi_{0}\in P$.   We define $Ef_{\theta} = \widehat{f_{\theta} d\sigma}$, where $f_{\theta}$ is $f$ restricted to $\theta$. We use a similar definition also when  $f$ is a function supported on the unit ball in $\RR^d$. We have $Ef=\sum_{\theta} Ef_{\theta}$. 

Now, we are ready to state our main decoupling type estimate.
\begin{thm}\label{bilinear decoupling}
Given $\lambda \geq 1$, $N_1\geq N_2\geq 1$. 
Let $f_1$ be supported on $P$ where $ |\xi| \sim 1$, and let $f_2$ be supported on $P$ where $ |\xi| \sim \frac{N_2}{N_1}$. Let $\Omega = \{ (t,x)\in [0, N_1^2]\times [0, (\lambda N_1)^2]^d\}$. For a finitely overlapping covering of the ball $B=\{|\xi|\leq 1\}$ of caps $\{\theta\}$, $|\theta|=\frac{1}{\lambda N_1}$, we have the following estimate. For any small $\epsilon >0$, 

\noindent
when $d=2$, 
\begin{equation}
\|Ef_{1} Ef_{2}\|_{L^2_{avg}(w_{\Omega})} \lesssim_{\epsilon} ( N_2)^{\epsilon}\lambda^{d/2} \left(\frac{1}{\lambda}+\frac{N_2^{d-1}}{N_1}\right)^{1/2}\prod_{j=1}^2\left(\sum_{|\theta|=\frac{1}{\lambda N_1}}\|Ef_{j,\theta}\|_{L^4_{avg}(w_{\Omega})}^2\right)^{1/2},
\end{equation}

\noindent
when $d\geq 3$,
\begin{equation}
\|Ef_{1} Ef_{2}\|_{L^2_{avg}(w_{\Omega})} \lesssim_{\epsilon} ( N_2)^{\epsilon}\lambda^{d/2} \left(\frac{N_{2}^{d-3}}{\lambda}+\frac{N_2^{d-1}}{N_1}\right)^{1/2}\prod_{j=1}^2\left(\sum_{|\theta|=\frac{1}{\lambda N_1}}\|Ef_{j,\theta}\|_{L^4_{avg}(w_{\Omega})}^2\right)^{1/2},
\end{equation}
 where $w_{\Omega}$ is a weight adapted to $\Omega$.  

\end{thm}
The presence of weight $w$ in these estimates is standard. We list the basic property of $w$ in Section \ref{subsectionweight}, and one can refer to \cite{bourgain2016study} for more details. The notation $L_{avg}(w_{\Omega})^{2}$ is explained in notation subsection below, subsection \ref{subsectionnota}.

The proof of Theorem \ref{bilinear decoupling}   gives another proof of the linear decoupling theorem in  \cite{bourgain2014proof} in dimension $d=2$, and does not rely on multilinear-Kakeya or multilinear restriction theorems in $\mathbb{R}^{3}$. The proof of Theorem \ref{bilinear decoupling} in dimension $d\geq 3$ relies instead on linear decoupling in $\RRR^{d+1}$, \cite{bourgain2014proof}.

\begin{rem}
The estimate in Theorem \ref{Main Theorem}, Theorem \ref{bilinear decoupling} is sharp up to an $N_{2}^{\epsilon}$.  See Section \ref{sharpness} for examples.
\end{rem}
\begin{rem}
The $N_{2}^{\epsilon}$ loss in Theorem \ref{Main Theorem} is typical if one wants to directly use a decoupling type argument. It may be possible to remove $N_{2}^{\epsilon}$ in the mass supercritical setting,  (in our case, this means $d\geq 3$), using the approach in  \cite{killip2016scale}, where  the scale invariant Strichartz estimates are studied.  
\end{rem}
\begin{rem}
Similar bilinear estimates for dimension $d\geq 3$ were also considered in \cite{killip2016scale} for non-rescaled tori, see Lemma 3.3.
On the other hand in this work we also consider the $d=2$ case which is  mass critical.  
\end{rem}
\subsection{Acknowledgment}
We thank Larry Guth for very helpful discussions during the course of this work.

\subsection{Background and motivation}
The system \eqref{lambda} and the bilinear estimates \eqref{main estimate} and \eqref{mainestimatehigh}  naturally appear in the study of the following nonlinear Schr\"odinger equation on the \textbf{non-rescaled} tori:
\begin{equation}\label{cubicnls}
\begin{cases}
iu_{t}+\Delta u=|u|^{2}u,\\
u(0)=u_{0 }\in H^{s}(\mathbb{T}^{d}).
\end{cases}
\end{equation}

Let us  focus for a moment on the   $d=2$ case. The Cauchy problem is said to be locally well-posed in $H^s(\mathbb{T}^d)$ if for any initial data $u_0 \in H^s(\mathbb{T}^d)$ there exists a time $T=T(\|u_0\|_s)$ such that a unique solution to the initial value problem exists on the time interval $[0, T]$.  We also require that the data to solution map is continuous from $H^s(\mathbb{T}^d)$ to $C^0_tH^s_x( [0,T] \times \mathbb{T}^d)$.  If $T=\infty$,  we say that a Cauchy problem is globally well-posed. 

The initial value problem  \eqref{cubicnls} is locally well-posed for initial data $u_{0}\in H^{s}, s>0$ via Strichartz estimates. Note that using iteration, by the energy conservation law, i.e. 		\begin{align*} 
		 E(u(t))=E(u_{0})=\frac{1}{2}\int |\nabla u|^{2}+\frac{1}{4}\int |u|^{4},
	\end{align*}
all initial data in $H^{1}(\mathbb{T}^{2})$  give rise to  a global solution.
Next, by the nowadays standard I-method, \cite{colliander2002almost}, by considering  a modified version of the energy, in the rational torus case, it was proved in  \cite{de2006global} that   \eqref{cubicnls} is indeed globally well-posedness for initial data in $H^{s}, s>2/3$. The key estimate there was in fact   \eqref{main estimate} for linear solutions on  \textbf{rescaled} tori, which we prove here to be available also for irrational tori.

The proof  for \eqref{main estimate} presented in \cite{de2006global} is only for rational tori since it relies on certain types of counting lemmata that  cannot directly work on irrational tori. One of the main purpose of this work in fact is to extend results on rational tori to irrational ones.

Based on the discussion we just made, as a corollary of Theorem \ref{Main Theorem}, we have
\begin{cor}\label{gwp}
The initial value problem  \eqref{cubicnls}  defined on any torus $\mathbb {T}^{2}$ is globally well-posed for initial data in $H^{s}(\mathbb {T}^{2})$ with $s>2/3$.\end{cor}
\begin{rem}
Results  such as  Corollary \ref{gwp} usually also give a control on the  growth of Sobolev norms of the global solutions. We do not address this particular question here. We instead refer the reader to the recent work \cite{Deng2016growth}.
\end{rem}

The original  Strichartz estimates needed to prove  the local well-posedness of Cauchy problems such as \eqref{cubicnls} were first obtained in \cite{bourgain1993fourier} via number theoretical  related counting arguments for rational tori. Recently, the striking proof of the $L^{2}$ decoupling Theorem, \cite{bourgain2014proof}, provided a completely different approach from which   all the desired Strichartz estimates on tori, both rational and irrational, follow. This approach in particular does not depend on counting lattice points. See also the work \cite{guo2014strichartz} and \cite{Deng2016str}.
The  method of proof we implement in this present  work is mostly inspired by \cite{bourgain2014proof} and the techniques used to prove the $L^{2}$ decoupling Theorem.

We quickly recall the main result in \cite{bourgain2014proof}. Let $P$ be a unit parabola in $\mathbb{R}^{d+1}$,  covered by finitely overlapping caps $\theta$ of  radius $\frac{1}{R}$. Let $f$ be a function defined on $P$, then one has for any $\epsilon>0$ small,
\begin{equation}\label{decouple}
\|Ef\|_{L^{p}(w_{B_{R^{2}}})}\lesssim_{\epsilon}  R^{\epsilon} (R^{2})^{d/4-\frac{d+2}{2p}}\left(\sum_{\theta} \|Ef_{\theta}\|_{L^{p}(B_{R^{2}})}^{2}\right)^{1/2}, \quad p\geq \frac{2(d+2)}{d}.
\end{equation}
Note that \eqref{decouple} corresponds to  Theorem 1.1  in \cite{bourgain2014proof}, and  the dimension $n$ in the estimate (2) there corresponds to  our $d+1$.
{Also note that the linear decoupling \eqref{decouple}  not only works for  those $f$ exactly supported on $P$, but those $f$ supported in a $R^{-2}$ neighborhood of $P$, and in this case, cap $\theta$ would be replaced by the $R^{-2}$ neighborhood of the original $\theta$, see Theorem 1.1 in \cite{bourgain2014proof}. }

 We remark that one key feature of this decoupling type estimate is that one needs to work on a larger scale in physical space, i.e. the scale  $R^{2}$ rather than $R$, in order to observe  the decoupling phenomena. The proper observational scale dictated by Heisenberg's uncertainty principle is $R$.

Indeed, one principle, which is usually called {\it parallel decoupling}, indicates that if decoupling happens in a small region, then decoupling happens in a large region as well. We state a bilinear version the parallel decoupling below.

\begin{lem}[\cite{bourgain2014proof}, \cite{bourgain2016study}]\label{parallel}
Let $D$ be a domain, and $D=D_{1}\cup D_{2}...\cup D_{J}$, $D_{i}\cap D_{j}=\emptyset$. If for some constant $A>0$ and  for function $h_{1}, h_{2},$ defined on the unit parabola, one has 
\begin{equation}
\|Eh_{1} Eh_{2}\|_{L_{avg}^2(w_{D_{i}})}\leq  A\prod_{j=1}^2\left(\sum_{|\theta|=\frac{1}{\lambda N_1}}\|Eh_{j,\theta}\|_{L^4_{avg}(w_{D_{i}})}^2\right)^{1/2} \, \,  i=1,\dots,J,
\end{equation}
then one also has
\begin{equation}
\|Eh_{1} Eh_{2}\|_{L_{avg}^2(w_{D})}\leq  A\prod_{j=1}^2\left(\sum_{|\theta|=\frac{1}{\lambda N_1}}\|Eh_{j,\theta}\|_{L^4_{avg}(w_{D})}^2\right)^{1/2}.
\end{equation}
\end{lem}

The proof of this particular formulation of parallel decoupling follows by Minkowski's inequality.

As it exists, parallel decoupling is a principle rather than a concrete lemma. We state the version here solely for concreteness. It should be easy to generalize the lemma under different conditions.

\subsection{Notation}\label{subsectionnota}

We write $A\lesssim B$ if $A\leq CB$, for a constant $C>0$, $A\sim B$ if both $A\lesssim B$ and $B\lesssim A$. We say $A\lesssim_{\epsilon} B$ if the constant $C$ depends on $\epsilon$. Similarly for $A\sim_{\epsilon} B$.  For a Borel set, $E\subset \mathbb{R}^d$, we denote that diameter of $E$ by $|E|$ and the Lebesgue measure of $E$ by $m(E)$.

We will use the usual function space $L^{p}$. We also use a (weighted) average version of $L^{p}$ space, i.e 
	\begin{align*}
		\|g\|_{L^p_{avg}(A)} =\left(\fint_A |g|^p\right)^{1/p}:= \left(\frac{1}{m(A)}\int_A|g|^p \right)^{1/p}
	\end{align*}
 and 
	\begin{align*}
		\|g\|_{L^p_{avg}(w_A)} =\left(\frac{1}{m(A)}\int|g|^p w_A \right)^{1/p},
	\end{align*}
where $w_A$ is a weight function described below.

For any function $f$, we use $\hat{f}$ to denote its Fourier transform.  When we say {\it unit ball}, we refer to a ball of radius  $r\sim 1$.  
We will often identify a torus as a bounded domain in Euclidean space, for example, we will view $(\mathbb{R}/\mathbb{Z})^{d}$ as $[0,1]^{d}\subset  \RRR^{d}$. 
In this work, $\Omega$ is used to denote the domain $[0,N_{1}^{2}]\times [0,(\lambda N_{1})^{2}]^{d} \subset \RRR^{d+1}$.

\subsection{  The weight $w_A$}\label{subsectionweight}

 If  $h$ is a Schwartz function whose Fourier transform, $\hat{h}$, is supported  in a ball of radius $1/R$, we expect $h$ be essentially constant on balls of radius $R$, and morally
	\begin{align}\label{intuitive}
	\|h\|_{L_{avg}^{p}(B_{R})} \sim \|h\|_{L^2_{avg}(B_{R})}\sim \|h\|_{L^{\infty}(B_{R})}.
	\end{align}

Expression \eqref{intuitive} is not rigorous, and the introduction of the weight $w_{B_{R}}$ is a standard way to overcome this technical difficulty.  We refer to Lemma 4.1 in \cite{bourgain2016study} for more detailed discussion of the weight function.

For any bounded open convex set $A$, the weight function $w_{A}$, might change from line to line, from the left hand side of the inequality to the right hand side, satisfies the same properties: 

\begin{itemize}
\item $\int w_{A} \sim m(A)$.
\item $w_{A}\gtrsim 1$ on $A$, and rapidly (polynomial type) decay outside $A$.
\end{itemize}

We will usually define $A$ to be a ball, or the product of balls in this paper. 

Furthermore, let $B_R$ be a ball centered at 0, and let $\mu_{B_{R}}$ be a function such that $\widehat{\mu_{B_{R}}}$ is about $\frac{1}{m(B_{1/R})}$ on $B_{1/R}$, and supported in $B_{2/R}$, then $\mu_{B_{R}}$ is about $1$ on $B_{R}$, decays faster than any polynomial outside of $B_{R}$. $\mu_{B_R}^2$ is positive, decays faster than any polynomial outside of $B_R$ and fourier supported in $B_{4/R}$, We take translations $B'$ of $B_R$ to cover the whole space, we note $\mu_{B'}$ as the corresponding translation of $\mu_{B_R}$ and  $w_{B_R}(B')=\max_{x\in B'} w_{B_R}$, we have the following useful property,
\begin{equation}\label{weight}
w_{B_R}(x)\leq \sum_{B'} w_{B_{R}}(B')1_{B'}(x) \lesssim \sum_{B'} w_{B_{R}}(B')\mu_{B'}^2(x)\lesssim w_{B_R}(x).
\end{equation}
The last inequality follows from the fact that $\mu_{B'}^2$ decays faster than any polynomial outside of $B'$. 

\begin{lem}\label{L infinity}
For a function $f$ supported in $B_{1/R}$, for any $p< \infty$, 
$$\|Ef\|_{L^{\infty}(B_{R})} \lesssim \|Ef\|_{L^p_{avg}(\mu_{B_{R}})}.$$ 
\end{lem}
We refer the proof to Corollary 4.3 in \cite{bourgain2016study} with the weight on the left hand side being $1_{B_R}$ so that on the right hand side we have a fast decay weight. 
{\begin{rem}
In general, Lemma~\ref{L infinity} should hold for any convex set $A$ and the dual convex body $A^*$.
\end{rem}
}

\section{Proof of Theorem \ref{Main Theorem} assuming Theorem \ref{bilinear decoupling} }
Assume Theorem~\ref{bilinear decoupling} , let us prove Theorem~\ref{Main Theorem}.  The argument  below  comes from the proof of  discrete restriction and Strichartz estimate on irrational tori assuming the $L^{2}$ decoupling estimate, see Theorem 2.2, Theorem 2.3 in  \cite{bourgain2014proof}. The argument originally comes  as observation due to Bourgain \cite{bourgain2013moment}. We record it here for completeness.

Let $\phi_{1}, \phi_{2}$ be as in Theorem \ref{Main Theorem}. We   rescale $\phi_{1}$ to be supported in the unit ball and rescale $\phi_{2}$ to be supported in a ball of radius $\sim\frac{N_2}{N_1}$.
Recall,
	\begin{equation}
		 U_{\lambda}(t)\phi_{j}(x,t) = \frac{1}{\lambda^{d/2}} \sum_{k\in \Lambda_{\lambda} , k\sim N_1}e^{2\pi i k\cdot x -|2\pi k |^2 t}\widehat{\phi}_{j}(k).
	\end{equation}

We perform a  change of variables $\xi=\frac{k}{N_1}$ and we 
let
	\begin{equation}
		 h_j(\tau) = \frac{1}{\lambda^{d/2}}\sum_{\xi \in \Lambda_{\lambda N_1}, |\xi|\sim 1} \widehat{\phi}_{j}(\xi N_1) \delta_{\xi}(\tau),\qquad j=1,2.
	\end{equation}

Note one can directly check that
	\begin{equation}
		 U_{\lambda}(t)\phi_{j} (x,t) = Eh_{j}(-2\pi N_1 x, (2\pi)^2 N_1^2 t).
	\end{equation}

Without loss of generality, we {suppress}  the constants $-2\pi$ and $(2\pi)^2$.  

 Let $Q_0=[0, N_{1}^2]\times { \mathbb{T}^d_{\lambda N_1}} $ and   let us view ${\mathbb{T}^d_{\lambda N_1}}$  as a compact set in $\RR^d$. In particular, one can construct the associated weight function $w_{Q_{0}}$.
 Direct computation (via change of variables) gives
	\begin{equation}\label{firstchangeofvariable}
		\|U_{\lambda}(t)\phi_1) U_{\lambda}(t)\phi_2\|_{L^2([0,1]\times \mathbb{T}^d_{\lambda})} \sim N_{1}^{-\frac{d+2}{2}} { m(Q_0)^{1/2}}\|Eh_{1} Eh_{2}\|_{L_{avg}^2(Q_{0})}
	\end{equation}
and due to the  the periodicity of $Eh_{i}, i=1,2$, one has 
	\begin{equation}
		\|Eh_{1} Eh_{2}\|_{L_{avg}^{2}(\Omega)}=\|Eh_{1} Eh_{2}\|_{L_{avg}^{2}(Q_{0})}.
	\end{equation}

For a covering $\{\theta\}$ of caps of radius $\frac{1}{\lambda N_1}$, each cap $\theta$ contains at most one $\xi_{\theta} \in \Lambda_{\lambda N_1}$, corresponding to $k_{\theta} =N_1 \xi_{\theta} \in \Lambda_{\lambda}$, then 
	\begin{align*}
		\|Eh_{j, \theta}\|_{L^4_{avg}(w_{Q_0})} \sim h_{j}(\xi_{\theta}) \sim \frac{1}{\lambda^d}\widehat{\phi}_{j}(k_{\theta})
	\end{align*}
and 
	\begin{align*}
		\prod_{j=1}^2\left(\sum_{|\theta|=\frac{1}{\lambda N_1}} \|Eh_{j}\|_{L^{4}_{avg}(w_{Q_0})}^2\right)^{1/2} & \sim \lambda^{-d} \prod_{j=1}^2\left(\frac{1}{\lambda^d}\sum_{k\in\Lambda_{\lambda}}|\widehat{\phi}_j(k)|^2\right)^{1/2}\\
		&\sim \lambda^{-d} \|\phi_1\|_{L^2}\|\phi_2\|_{L^2}.
	\end{align*}
For convenience of notation let 
	\begin{equation}\label{notationconvience}
		 D_{\lambda, N_{1},N_{2}}:=
		\begin{cases}
			\frac{1}{\lambda}+\frac{N_{2}}{N_{1}}, \text{ when } d=2,\\
			\frac{N_{2}^{d-3}}{\lambda}+\frac{N_{2}^{d-1}}{N_{1}}, \text{  when } d\geq 3.
		\end{cases}
	\end{equation}

Recall that  $\Omega=[0,N_{1}]^{2}\times [0,(\lambda N_{1})^{2}]^{d}$, we apply Theorem \ref{bilinear decoupling} with $f_{j}= h_{j}$, and we have 
	\begin{equation}\label{Omega domain}
		\|Eh_{1} Eh_{2}\|_{L_{avg}^2(w_{\Omega})} \lesssim_{\epsilon} (N_2)^{\epsilon}\lambda^{d/2} D_{\lambda, N_{1}, N_{2}}^{1/2}\prod_{j=1}^2\left(\sum_{|\theta|=\frac{1}{\lambda N_1}}\|Eh_{j,\theta}\|_{L^4_{avg}(w_{\Omega})}^2\right)^{1/2}.
	\end{equation}
Note that $\Omega$ can be covered   by ${Q}$ such that $\{Q\}$ are finitely overlapping and each $Q$ is a translation of $Q_{0}$. 
Since $Eh_{j}$ are periodic on $x$, estimate \eqref{Omega domain} is equivalent to 
	\begin{equation}\label{domain}
		\|Eh_{1} Eh_{2}\|_{L^2_{avg}(w_{Q_0})} \lesssim_{\epsilon} ( N_2)^{\epsilon}\lambda^{d/2}D_{\lambda, N_{1}, N_{2}}^{1/2}\prod_{j=1}^2\left(\sum_{|\theta|=\frac{1}{\lambda N_1}}\|Eh_{j,\theta}\|_{L^4(w_{Q_0})}^2\right)^{1/2}.
	\end{equation}

Plugging  \eqref{domain} into \eqref{firstchangeofvariable} gives
	\begin{align*}
		&\|U_{\lambda}(t)\phi_1 U_{\lambda}(t)\phi_2\|_{L^2([0,1]\times \mathbb{T}^d_{\lambda})}\\& \lesssim N_1^{-\frac{d+2}{2}}\cdot N_{1} { m(\mathbb{T}^d_{\lambda N_1})^{1/2}}\lambda^{-d}\cdot (N_2)^{\epsilon}\lambda^{d/2}D_{\lambda, N_{1}, N_{2}}^{1/2}\|\phi_1\|_{L^2}\|\phi_2\|_{L^2}\\
		&\sim ( N_2)^{\epsilon}D_{\lambda, N_{1}, N_{2}}^{1/2}\|\phi_1\|_{L^2}\|\phi_2\|_{L^2}
	\end{align*}
and Theorem \ref{Main Theorem} follows.
\medskip

The rest of the paper details the proof of Theorem \ref{bilinear decoupling}.

\section{An overview of the proof of Theorem \ref{bilinear decoupling}}

First, we reduce the proof of Theorem \ref{bilinear decoupling} to the following proposition.
\begin{prop}\label{transversal}
Let $\tau_1$ be a cap of radius $\frac{N_2}{N_1}$ supported at $\xi$ and $|\xi|\sim 1$. Let $\tau_2$ be a cap of radius $\frac{N_2}{N_1}$ supported at $\xi$ with $|\xi|\sim \frac{N_2}{N_1}$. Let $f_{j}$ be a function supported in $\tau_j$, then for any small $\epsilon >0$, 

\noindent
when $d=2$
\begin{equation}
\|Ef_1 Ef_2\|_{L^2_{avg}(w_{\Omega})} \lesssim_{\epsilon} ( N_2)^{\epsilon} \lambda^{d/2} \left(\frac{1}{\lambda}+\frac{N_2}{N_1}\right)^{1/2} \prod_{j=1}^2 \left(\sum_{|\theta|=\frac{1}{\lambda N_1}, \theta \subset\tau_j} \|Ef_{j,\theta}\|_{L^4_{avg}(w_{\Omega})}^2\right)^{1/2}
\end{equation}

\noindent
when $d\geq 3$,
\begin{equation}
\|Ef_1 Ef_2\|_{L^2_{avg}(w_{\Omega})} \lesssim_{\epsilon} ( N_2)^{\epsilon} \lambda^{d/2} \left(\frac{N_{2}^{d-3}}{\lambda}+\frac{N_2^{d-1}}{N_1}\right)^{1/2} \prod_{j=1}^2 \left(\sum_{|\theta|=\frac{1}{\lambda N_1}, \theta \subset\tau_j} \|Ef_{j,\theta}\|_{L^4_{avg}(w_{\Omega})}^2\right)^{1/2}.
\end{equation}
\end{prop}

Now, let $f_{1}, f_{2}$ be as in Proposition \ref{transversal}. We define  $K_{0}(\lambda, N_{1},N_{2})$ to be  the best constant such that
\begin{equation}\label{best}
 \|Ef_{1} Ef_{2}\|_{L^2_{avg}(w_{\Omega})} \leq \lambda^{d/2} K_{0}(\lambda, N_1, N_2) \prod_{j=1}^2\left(\sum_{|\theta|=\frac{1
 }{\lambda N_1}}\|Ef_{j,\theta}\|_{L^4_{avg}(w_{\Omega})}^2\right)^{1/2}.
\end{equation}
We also let $\tilde{K}(\lambda, N_{1}, N_{2})$ and  $K(\lambda, N_{1},N_{2})$ be defined as the best constants such that
\begin{equation}\label{best3}
 \|Ef_{1} Ef_{2}\|_{L^2_{avg}(w_{[0,N_{1}^{2}]\times [0,\lambda N_{1}]^{d}})} \leq \lambda^{d/2} \tilde{K}(\lambda, N_1, N_2) \prod_{j=1}^2\left(\sum_{|\theta|=\frac{1
 }{\lambda N_1}}\|Ef_{j,\theta}\|_{L^4_{avg}(w_{[0,N_{1}^{2}]\times [0,\lambda N_{1}]^{d}})}^2\right)^{1/2},
\end{equation}
\begin{equation}\label{best2}
 \|Ef_{1} Ef_{2}\|_{L^2_{avg}(w_{B_{N_{1}^{2}}})} \leq \lambda^{d/2} K(\lambda, N_1, N_2) \prod_{j=1}^2\left(\sum_{|\theta|=\frac{1
 }{\lambda N_1}}\|Ef_{j,\theta}\|_{L^4_{avg}(w_{B_{N_{1}^{2}}})}^2\right)^{1/2}.
\end{equation}
Below we will prove that 
\begin{equation}\label{uuu}
\begin{aligned}
&K_{0}(\lambda, N_1, N_2)\lesssim N_2^{\epsilon}(\frac{1}{\lambda}+\frac{N_2}{N_1})^{1/2}, \qquad d=2,\\
&K_{0}(\lambda, N_1, N_2)\lesssim N_2^{\epsilon}(\frac{N_{2}^{d-3}}{\lambda}+\frac{N_2^{d-1}}{N_1})^{1/2},  \qquad d\geq 3.
\end{aligned}
\end{equation} 
We point out here that  by parallel decoupling and Lemma \ref{parallel}  one always has
\begin{equation}\label{tn}
K_{0}(\lambda, N_{1},N_{2})\lesssim K(\lambda, N_{1},N_{2}), \hspace{1cm} K_{0}(\lambda, N_{1}, N_{2})\lesssim \tilde{K}(\lambda, N_{1}, N_{2}).
\end{equation}

The proof of Proposition \ref{transversal} or equivalently \eqref{uuu}  proceeds as follows.
We  first show
\begin{lem}\label{big lambda}
When $\lambda \geq N_1$, 
\begin{equation}
\tilde{K}(\lambda, N_1, N_2) \lesssim  N_{2}^{\epsilon}\frac{N_2^{(d-1)/2}}{N_1^{1/2}}.
\end{equation}
\end{lem}
Note that when $\lambda \geq N_{1}$, Proposition \ref{transversal} follows from \eqref{tn}  and Lemma \ref{big lambda}.

Then, we  show 
\begin{lem}\label{criticallll}
When $\lambda \leq N_{1}$, 
\begin{equation}
\begin{aligned}
&K (\lambda,N_{1},N_{2} )\lesssim N_{2}^{\epsilon}(\frac{1}{\lambda}+\frac{N_{2}}{N_{1}})^{1/2}, \qquad d=2,\\
&K (\lambda,N_{1},N_{2} )\lesssim N_{2}^{\epsilon}(\frac{N_{2}^{d-3}}{\lambda}+\frac{N^{d-1}_{2}}{N_{1}})^{1/2},  \qquad d=3.
\end{aligned}
\end{equation}
\end{lem}
From  \eqref{tn}, clearly Proposition \ref{transversal} follows from Lemma \ref{big lambda} and Lemma \ref{criticallll}.

\medskip
The proof of Lemma \ref{criticallll} in dimension $d=2$ relies on  induction (of scale $N_{2}$). The proof of Lemma \ref{criticallll} in dimension in $d\geq 3$ is easier and more straightforward, (in some sense, it also relies on induction, but it is enough to induct only once.)

\medskip
We  first show the base case: 
\begin{lem}\label{base case}
When  $\lambda \leq N_1$ and $N_2\lesssim 1$, , $K(\lambda, N_1, N_2)  \lesssim \frac{1}{\lambda^{1/2}}$.
\end{lem}
Lemma \ref{base case} is not as useful in dimension $d\geq 3$, we indeed have a better estimate: 
\begin{lem}\label{base case high}
When $d\geq 3$,  $\lambda \leq N_{1}$ and $\lambda \leq \frac{ N_{1}}{N_{2}^{2}}$, $K(\lambda, N_{1}, N_{2})\lesssim \left(\frac{N_{2}^{d-3}}{\lambda}\right)^{1/2}$.
\end{lem}

We then show the following lemma, which ensures that we  only need to induct until $\lambda\leq  \frac{N_1}{N_2}$, when $d=2$, and until $\frac{N_{1}}{N_{2}}$ when $d\geq 3$. 
\begin{lem}\label{InductionN2}
Let $\lambda\leq N_{1}$.

\medskip
Let  $d=2$.   Assume we have that $K(\lambda, N_1, N_2) \leq \lambda^{-1/2}$ when $\lambda < \frac{N_1}{N_2}$. Then 
	\begin{align*}
		K(\lambda, N_1, N_2) \leq N_{2}^{\epsilon}\frac{N_2^{\frac{d-1}{2}}}{N_1^{1/2}} \qquad \mbox{ when}\quad  \lambda \geq \frac{N_1}{N_2}.
	\end{align*}

\medskip
Let $d\geq 3$. Assume we have that $K(\lambda, N_{1},N_{2})\leq (\frac{N_{2}^{d-3}}{\lambda})^{1/2}$ when $\lambda <\frac{N_{1}}{N_{2}^{2}}$. Then
	\begin{align*}
		K(\lambda, N_1, N_2) \leq N_{2}^{\epsilon}\frac{N_2^{\frac{d-1}{2}}}{N_1^{1/2}} \quad \mbox{ when} \quad  \lambda \geq \frac{N_1}{N^{2}_2}.
	\end{align*}
\end{lem}
Note that when $d\geq 3$, Lemma \ref{base case high} and Lemma \ref{InductionN2}  imply  Lemma \ref{criticallll}.
In dimension $d=2$,  we use induction (we rely on the so-called parabolic rescaling) to finish the proof of Lemma \ref{criticallll}.

\medskip
We end this section with an outline of the  structure of the rest  of the paper. We show that Proposition \ref{transversal} implies Theorem \ref{bilinear decoupling} in  Section \ref{sectionred}. Lemma \ref{big lambda}, Lemma \ref{base case}, Lemma \ref{InductionN2} all rely  on the exploration of the so-called {\it transversality}  which essentially allow us to reduce  the dimensionality of the problem. We first explore  {\it transversality} in Section \ref{sectionpro} and then we prove Lemma \ref{big lambda}, Lemma \ref{base case}, Lemma \ref{InductionN2} in Section \ref{section3lemma}.

\medskip
The detail of the induction procedure, (which is non trivial), that is used  to prove  Lemma \ref{criticallll} in dimension $d=2$ is  given in Section \ref{sectionpro}.
We remark here the proof of Lemma \ref{criticallll} relies on Lemma \ref{big lambda}.

\medskip
Finally, we  prove  Lemma \ref{base case high} at the  end of Section \ref{sectionind}, which, together with Lemma \ref{InductionN2} will  conclude the proof of Lemma \ref{criticallll} in dimension $d\geq 3$.

\section{Proposition \ref{transversal} implies Theorem \ref{bilinear decoupling}}\label{sectionred}

We first introduce one standard but important tool in the following lemma.

\begin{lem}\label{L2 orthogonality}[\cite{bourgain2014proof}, \cite{bourgain2016study}]
Let $\{g_{\alpha}\}$ be a family of functions such that  $\mbox{supp } \widehat{g}_{\alpha}$ are finitely overlapped cubes of length $\rho$.  Let  {$A$} be bounded convex open set tiled by finitely overlapped cubes { $Q$} of side length $\geq \rho^{-1}$, then for the $w_{A}$ adapted to $A$, the following holds,
	\begin{align*} 
		\fint_{A} |\sum g_{\alpha}|^2w_{A}  \lesssim \sum\frac{1}{m(A)} \int |g_{\alpha}|^2 w_A.
	\end{align*}
\end{lem}

\begin{proof}
Since we can sum up the weight function over a finitely overlapping cover $\{Q\}$ of $A$: $w_{A}=\sum_{Q\subset A} w_{Q}$, it suffices to prove for $A=Q$. 
Recall by the inequality~\ref{weight}, we cover the whole space $\mathbb{R}^n$ by translations $Q'$ of $Q$, 

	\begin{align*}
		\fint_{Q} |\sum g_{\alpha}|^2w_{Q}dx&\leq\frac{1}{m(Q)} \sum_{Q'} \int_{Q'} w_{Q}(Q') |\sum g_{\alpha} |^2 \\
		&\leq \frac{1}{m(Q)} \sum_{Q'}w_{Q}(Q')\int |\sum g_{\alpha}|^2 \mu_{Q'}^2 \\
		&=  \frac{1}{m(Q)}\sum_{Q'} w_{Q}(Q')\int|\widehat{g}_{\alpha}\ast\widehat{\mu_{Q'}}|^2  \\
		&\lesssim \frac{1}{m(Q)}\sum_{Q'}w_{Q}(Q') \sum_{\alpha}\int |g_{\alpha}|^2 \mu_{Q'}^2\\
		&\lesssim \frac{1}{m(Q)}\sum_{\alpha}\int |g_{\alpha}|^2 w_{Q}
	\end{align*}
\end{proof}

Now we can reduce Theorem \ref{transversal} to  a bilinear decoupling on two $\frac{N_2}{N_1}$-diameter caps.

\begin{lem}\label{reduce}
Theorem~\ref{bilinear decoupling} is equivalent to Proposition~\ref{transversal}.
\end{lem}
\begin{proof}
Let $f_{1}, f_{2}$ be as in Theorem \ref{bilinear decoupling}.
Then $f_1 =\sum_{|\tau|=\frac{N_2}{N_1}} f_{1,\tau}$ and  $f_{1,\tau}$'s are supported on finitely overlapping caps of diameter $\frac{N_2}{N_1}$. 

 Since $|f_2|$ is supported in a cap of diameter $\frac{N_2}{N_1}$, the supports of $\{\widehat{Ef}_{1,\tau} \ast \widehat{Ef_2}\}_{\tau}$ are in finitely overlapping cubes of length $\frac{N_2}{N_1}$. Since the scale of $\Omega$ is larger than $N_{1}/N_{2}$, i.e. it contains a ball of radius $>N_{1}/N_{2}$, By Lemma \ref{L2 orthogonality}, 
	\begin{align*}
		\fint_{\Omega}|Ef_1 Ef_2|^2w_{\Omega} dx & \leq\sum_{|\tau|=\frac{N_2}{N_1}} \left| \fint_{\Omega} Ef_{1,\tau}Ef_2\right|^2 w_{\Omega} dx\\
	\end{align*}
Now apply Proposition \ref{transversal} for $f_{1,\tau}$ and $ f_{2}$ for each $\tau$, Theorem \ref{bilinear decoupling} follows. 
\end{proof}

\section{{ Transversality}}\label{sectionpro}
Let $f_{1}, f_{2}$ be as in Proposition \ref{transversal}, then $f_{1}$ is supported around $(0,0,\dots,0,1,1)$ and  $f_{2}$ is supported around $(0,0,\dots,0)$. The main goal of this section is to explore the transversality between $(0,0,\dots,0,1)$ and $(0,0,0,\dots,0)$, or more precisely, the transversality between the unit normal vectors of the truncated parabola at these two points. The main lemma in this section  is  Lemma \ref{projection to 2D} below, and Corollary  \ref{Tcor} which essentially follows from Lemma \ref{projection to 2D}.

We first introduce some basic notation. Let $(e_{1},\dots,e_{d})$ be the standard basis of $\RRR^{d}$. We will encounter caps of radius $v$ around $(0,0,\dots,0)$ and $(0,\dots,0,1,1)$ on the parabola. Note around those two points, when $v$ is small (which is always the case in our work),  one may view those caps  as their natural  projection to $\RRR^{d-1}$.  And their image is essentially a square/cap of radius $v$. {  We say that a $(v,v^{2})$-plate is a $d$-dimensional rectangle with the short side on $e_{d-1}$ direction such that its image under under the orthogonal projection to $R^{d-1}$ is a $v\times v\times \cdots \times v \times v^{2}$-rectangle.}

\begin{lem}\label{projection to 2D}
Given $|\upsilon|< 1$, let $f_1$ be a function supported on a cap of radius $\upsilon$, centered at $(0,\dots, 0, 1,1)$ on the truncated parabola $P$,  and let $f_2$ be a function supported on a cap of radius $\upsilon$ centered at $(0,\dots,0,0,0)$ on the paraboloid. For a covering $\{\tau_i\}$ of $supp \, \, f_i$ with $(\upsilon,\upsilon^2)$--plates, with the shorter side on $e_{d-1}$ direction. We have the following decoupling inequality,  for any $R> \upsilon^{-2}$, 
	\begin{equation}\label{1ddecouple}
		\int |Ef_1 Ef_2|^2 w_{B_R}  \lesssim \sum_{\tau_1, \tau_2}\int |Ef_{1,\tau_1} Ef_{2,\tau_2}|^2 w_{B_R}. 
	\end{equation}
\end{lem}

\begin{rem} We thank J. Ramos for pointing out that Lemma \ref{projection to 2D} is a particular case of Proposition 2 in his  work \cite{ramos2016trilinear}.  We still write a proof in this paper for clarity.
\end{rem}

\begin{proof}
The proof is similar to the proof of the $L^{4}$ Strichartz estimate on the one  dimensional torus. 
From the inequality~\ref{weight}, we only need to prove that $$\int_{B'} |Ef_1 Ef_2|^2   \lesssim \sum_{\tau_1, \tau_2}\int |Ef_{1,\tau_1} Ef_{2,\tau_2}|^2 \mu_{B'}^2$$ for all translation $B'$ of $B_R$.

\begin{equation}\label{tempproject}
\int_{B'} |Ef_{1}Ef_{2}|^{2}\leq  \sum_{\tau_{1},\tau_{2},\tau_{3},\tau_{4}}\int_{B'}  Ef_{1,\tau_{1}}Ef_{2,\tau_{2}}\overline{Ef_{1,\tau_{3}}} \overline{Ef_{2,\tau_{4}}}\mu_{B'}^2
\end{equation}

Let $\xi_{i}\in \tau_{i}$, $\xi_{i}=(\xi_{i,1},..,\xi_{i,d-1}, \sum_{j=1}^{d-1}(\xi^{j}_{i})^{2})\equiv (\bar{\xi}_{i},\xi_{i,d-1},|\bar{\xi}_{i}|^{2}+(\xi_{i}^{d-1})^{2})$, $i=1,2,3,4$. 
We have 
	\begin{align}\label{t_{0}} 
		\begin{array}{ll}
			|\bar{\xi}_{i}|\lesssim \upsilon, \hspace{.5cm}& i=1,2,3,4.\\
			|\xi_{i,d-1}-1|\lesssim \upsilon, \hspace{.5cm}&i=1,3.\\
			|\xi_{i,d-1}|\lesssim \upsilon, \hspace{.5cm}&i=2,4.
		\end{array}
	\end{align}

 Essentially, for any $\tau_{1},\tau_{2},\tau_{3},\tau_{4}$ such that $$\int Ef_{1,\tau_{1}}Ef_{2,\tau_{2}}\overline{Ef_{1,\tau_{3}}} \overline{Ef_{2,\tau_{4}}}\mu_{B'}^2\neq 0,$$ one must have for some $\xi_{i}\in \tau_{i}$,
	\begin{equation}\label{t1}
		\begin{aligned}
			\xi_{1}-\xi_{3}=\xi_{2}-\xi_{4}+O(R^{-1}),\\
			|\xi_{1}|^{2}-|\xi_{3}|^{2}=|\xi|_{2}^{2}-|\xi|_{4}^{2}+O(R^{-1}),\\
		\end{aligned}
	\end{equation}
and the second formula in \eqref{t1} implies
	\begin{equation}
		(\xi_{1,d-1}-\xi_{3,d-1})(\xi_{1,d-1}+\xi_{3,d-1})=O(|\xi_{2}|^{2}+|\xi_{4}|^{2})+O(|\bar{\xi}_{1}|^{2}+|\bar{\xi}_{3}|^{2})+O(R^{-1}).
	\end{equation}
Plugging into \eqref{t_{0}}, one has $|\xi_{1,d-1}-\xi_{3,d-1}|\lesssim v^{2}$, which again implies $|\xi_{2,d-1}-\xi_{4,d-1}|\lesssim v^{2}$.

To summarize, $\int Ef_{1,\tau_{1}}Ef_{2,\tau_{2}}\overline{Ef_{1,\tau_{3}}} \overline{Ef_{2,\tau_{4}}}\mu_{B'}^2\neq 0$ implies the distance between $\tau_{1}$ and $\tau_{3}$ and the distance between $\tau_{2}$ and $\tau_{4}$ are both bounded by $v^{2}$, which essentially means $\tau_{i}=\tau_{i+2}$, $i=1,2$. Applying this fact to \eqref{tempproject}, Lemma \ref{projection to 2D} follows.
\end{proof}

\begin{rem}\label{Ktran}
A quantitative version of estimate \eqref{1ddecouple} can be stated as follows:   assume that the support of $f_{1}$ is centered at $(0,1/K, (1/K)^{2})$ rather than $(0,0,1)$, from the proof we can attain the same estimate as in \eqref{1ddecouple} by introducing an additional constant $K$,
	\begin{align}\label{1ddecouplemode}
		\int |Ef_1 Ef_2|^2 w_{B_R}  \lesssim  K\sum_{\tau_1, \tau_2}\int |Ef_{1,\tau_1} Ef_{2,\tau_2}|^2 w_{B_R} 	
	\end{align}
Indeed, the proof  essentially only relies on the fact that for $\xi_{i}\in \mbox{supp } f_{i}, i=1,2$, the difference between the $d-1$ components is at least $\frac{1}{K}$. Similar arguments also hold for estimate in Lemma \ref{Testimate},  Cor \ref{Tcor} below.
\end{rem}

\begin{rem}\label{p2drem}
We remark that for  any $\alpha<\upsilon$, a function which is supported on a   cap of radius  $\alpha$ can be naturally understood as a function supported on a  cap of radius $\upsilon.$
\end{rem}

Lemma \ref{projection to 2D} facilitates the decomposition of caps of radius $v$ into plates of size $(v,v^{2})$, we can further decompose those into caps of radius $v^{2}$.
\begin{lem}\label{Testimate}
With same notation as in Lemma~\ref{projection to 2D},  $R\geq \upsilon^{-2}$, let $supp f_{i}$ be the covered by finitely overlapping caps $\theta_{i}$ of radius $v^{2}$, $i=1,2$. Then 
	\begin{equation}\label{formulatestimate}
		\int |Ef_{1}Ef_2|^2w_{B_R} \lesssim \upsilon^{-(d-1)} \sum_{|\theta_i|=\upsilon^2}\int|Ef_{1,\theta_1}Ef_{2,\theta_2}|^2w_{B_R}.
	\end{equation}
\end{lem}

\begin{proof}
Clearly, we need only to prove \eqref{formulatestimate} for every ball of radius $\upsilon^{-2}$ contained in $B_{R}$, and then sum them together. (This is in the same principle of parallel decoupling, Lemma \ref{parallel}.)

Fix a pair of $(\upsilon, \upsilon^2)$--plates  $\tau_1, \tau_2$. 
\begin{equation}\label{hhhh}
\begin{aligned}
\int |Ef_{1,\tau_1} Ef_{2,\tau_2}|^2 w_{B_R} &= \int |\sum_{\theta_{2}\subset \tau_2, |\theta_2|=\upsilon^2} Ef_{1,\tau_1}Ef_{2, \theta_2}|^2 w_{B_R} \\
&\leq \upsilon^{-(d-1)} \sum_{\theta_2\subset \tau_2, |\theta_2|=\upsilon^2}|Ef_{1,\tau_1}Ef_{2,\theta_2}|^2 w_{B_R}\\
&\lesssim \sum_{\theta_j\subset \tau_j, |\theta_j|=\upsilon^2}|Ef_{1,\theta_1}Ef_{2,\theta_2}|^2 w_{B_R}
\end{aligned}
\end{equation}
The last inequality follows from Lemma~\ref{L2 orthogonality} and Lemma~\ref{reduce}.

\end{proof}

\begin{rem}\label{easyrem}
Similar to Remark \ref{p2drem}, {for $\upsilon^{2}<\alpha<\upsilon$,  a cap of scale $\upsilon$ naturally lies in a cap of scale $\sqrt{\alpha}$.} Thus if we let  $f_1$ be a function supported on a cap of radius $\alpha$, centered at $(0,\dots, 0, 1,1)$ on the paraboloid and we let  $f_2$ be a function supported on a cap of radius $\alpha$ centered at $(0,\dots,0,0,0)$ on the paraboloid, then by arguing similar to the proof of Lemma \ref{Testimate}, we have for $R\geq 
\alpha^{-1}$, 
	\begin{equation}\label{formulatestimatemode}
		\int |Ef_{1}Ef_2|^2w_{B_R} \lesssim (\upsilon/\alpha)^{(d-1)} \sum_{|\theta_i|=\alpha}\int|Ef_{1,\theta_1}Ef_{2,\theta_2}|^2w_{B_R}.
	\end{equation}

\end{rem}

If we directly use Holder inequality for all caps in the support of $f_{i}$ to estimate as in \eqref{hhhh}, then the interpolation in the proof of Lemma \ref{Testimate} will give us a constant $v^{-d}$ rather than $v^{-(d-1)}$ in \eqref{formulatestimate}, since one has  $v^{-d}$ caps for each $f_{i}$. The bilinear transversality, i.e. the transversality between $(0,0,\dots,0)$ and $(0,\dots,0,1,1)$ helps in reducing the dimension by one since  in one direction we can use $L^{4}$ orthogonality, as shown in Lemma \ref{projection to 2D}. Thus here we are able to improve the constant in \eqref{formulatestimate} to $v^{-(d-1)}$.

\begin{cor}\label{Tcor}
Same notation as in Lemma~\ref{projection to 2D}, there exists a constant $C$, such that for any $\upsilon$, $\delta$, $R^{-1}\leq \delta \leq \upsilon$, 
	\begin{align*}
		\int |Ef_{1}Ef_2|^2 w_{B_R} \lesssim   \left(\frac{\upsilon}{\delta}\right)^{d-1}\Big|\frac{\log \delta}{\log \upsilon}\Big|^{C} \sum_{|\theta_i|=\delta}\int |Ef_{1,\theta_1}Ef_{2,\theta_2}|^2 w_{B_R}.
	\end{align*}
\end{cor}

\begin{proof}
The proof is most clear when $\delta=\upsilon^{2^{n}}$ for some $n$, let us first handle this case and then go to the general case.
One may use  induction. (This induction, however, does not rely on parabolic rescaling.)  If $n=0$, there is nothing to prove. 

Assume the result holds for the case  $n=k$, let us turn to the  case $n=k+1$, where  $\delta=v^{2^{k+1}}, \, \delta^{1/2}=v^{2^{k}}$, thus by induction assumption, we have
	\begin{equation}\label{f1}
		\int |Ef_{1}Ef_2|^2 w_{B_R} \lesssim    \left(\frac{\upsilon}{\delta^{1/2}}\right)^{d-1}2^{Ck} \sum_{|\eta_i|=\delta^{1/2}}\int |Ef_{1,\eta_1}Ef_{2,\eta_2}|^2 w_{B_R}.
	\end{equation}
Now note $R\geq (\delta^{-1/2})^{2}$, by Lemma \ref{Testimate}, we have for each pair $(\eta_{1}, \eta_{2})$ in \eqref{f1} that 
	\begin{equation}\label{f2}
		\int |Ef_{1,\eta_1}Ef_{2,\eta_2}|^2 w_{B_R} \lesssim (\delta^{1/2})^{-(d-1)}\sum_{\theta_{i}\subset\eta_{i}, |\theta_{i}|=\delta}\int |Ef_{1,\theta_1}Ef_{2,\theta_2}|^2 w_{B_R}.
	\end{equation}
The case $n=k+1$ clearly follows if one plugs \eqref{f2} into \eqref{f1}, taking  the constant $C$ large enough.

Now we turn to the general case, we only need to work on the case $\vv^{2^{n+1}}<\delta<\vv^{2^{n}}$.
Recall that previously, when $\delta=\vv^{2^{n}}$, we used induction as $\vv\rightarrow v^{2}\rightarrow \vv^{2^{2}}\dots\rightarrow \vv^{2^{n}}=\delta$, and in each step we used Lemma \ref{Testimate} to finish the induction $\vv^{2^{k}}\rightarrow \vv^{2^{k+1}}$.

 In the case  $\vv^{2^{n+1}}<\delta<\vv^{2^{n}}$ we have $\vv^{2^{n}}<\delta^{1/2}$,  and we use induction  as  before for 
$\vv\rightarrow \vv^{2}\rightarrow \vv^{2^{2}}\dots\rightarrow \vv^{2^{n}},$ and we use \eqref{formulatestimatemode} to use induction again  from $\vv^{2^{n}}$ to $\delta$. This ends the proof.
\end{proof}

\section{Proof of Lemma \ref{big lambda}, Lemma \ref{base case} and  Lemma \ref{InductionN2}}\label{section3lemma}

We are now prepared to use transversality to prove Lemma \ref{big lambda}, Lemma \ref{base case}, and  Lemma \ref{InductionN2}.
Recall Lemma \ref{big lambda} concerns $\tilde{K}(\lambda, N_{1}, N_{2})$ defined in \eqref{best3}. Furthermore, Lemma \ref{base case} and Lemma \ref{InductionN2} refer to $K(\lambda, N_{1}, N_{2})$ defined in \eqref{best2}.

\subsection{Proof of Lemma \ref{big lambda}}\label{ss1}
For convenience of notation, we let $\Omega_{1}:=[0, N_{1}^{2}]\times [0, \lambda N_{1}]^{d}$. Note that one can use finite overlapped balls of radius $N_{1}^{2}$ to cover $\Omega_{1}$ since $\lambda \geq N_{1}$. We want to prove 
\begin{equation}\label{de1}
\|Ef_{1} Ef_{2}\|_{L^2_{avg}(\omega_{\Omega_{1}})} \lesssim_{\epsilon} \lambda^{d/2}N_{2}^{\epsilon}\frac{N_{2}^{d-1}}{N_{1}}\prod_{j=1}^2\left(\sum_{|\theta|=\frac{1
 }{\lambda N_1}}\|Ef_{j,\theta}\|_{L^4_{avg}(w_{\Omega_{1}})}^2\right)^{1/2}.
\end{equation}
We first apply Corollary \ref{Tcor} with $\delta= N_1^{-2}$ , $\upsilon = \frac{N_2}{N_1}$, $R= N_{1}^{2}$.   Note that  $\delta\leq v$. Then we have 
\begin{equation}\label{1111}
\begin{aligned}
\int |Ef_1 Ef_2|^2w_{B_{ N^{2}_{1}}}&\lesssim (N_1N_2)^{d-1}\Big|\frac{\log N_1}{\log N_1 -\log  N_2}\Big|^{C}  \sum_{|\theta_j|= \frac{1}{N_1^2}} \int |Ef_{1, \theta_1} Ef_{2,\theta_2}|^2 w_{B_{N^{2}_{1}}}\\
&\lesssim (N_1 N_2)^{d-1}N_{2}^{\epsilon} \sum_{|\theta_j|= \frac{1}{N_1^2}} \prod_{j=1}^2\|Ef_{j,\theta_j}\|_{L^4(w_{B_{ N^{2}_{1}}})}^2. 
\end{aligned}
\end{equation}
\begin{rem}\label{rrr}

We avoid the case when $N_1=N_2$, and thus  $\ln N_{1}-\ln N_{2}=0$,   by first  decomposing caps of diameter $N_{2}/N_{1}$ into caps of  diameter $N_{2}/2N_{1}$ with loss of a fixed constant, then continuing with the proof as above. In all of the text that follows, one may assume, without loss of generality, that  $N_{1}\geq 2N_{2}$.

\end{rem}

Via  the principle of parallel decoupling, Lemma \ref{parallel}, or  by  summing different $B_{N_{1}^{2}}$ together,  we have 
\begin{equation}\label{1dddd}
\begin{aligned}
\int |Ef_1 Ef_2|^2 w_{\Omega_{1}}&
&\lesssim (N_1 N_2)^{d-1}N_{2}^{\epsilon}  \sum_{|\theta_j|= \frac{1}{N_1^2}} \prod_{j=1}^2\|Ef_{j,\theta_j}\|_{L^4(w_{\Omega_{1})}}^2. 
\end{aligned}
\end{equation}

Next we would like to show that 
\begin{equation}\label{min}
\|Ef_{j,\theta_j}\|_{L^4(w_{\Omega_1})}^2 \leq (\frac{\lambda}{N_1})^{d/2}\sum_{\theta_j' \subset \theta_j, |\theta_j'|=\frac{1}{\lambda N_1}} \|Ef_{j,\theta_j'}\|_{L^4(w_{\Omega_1})}^2.
\end{equation}
It suffices to show 
$$\|Ef_{j,\theta_j}\|_{L^4_{avg}(\Omega_1)}^2 \leq (\frac{\lambda}{N_1})^{d/2}\sum_{\theta_j' \subset \theta_j, |\theta_j'|=\frac{1}{\lambda N_1}} \|Ef_{j,\theta_j'}\|_{L^4_{avg}(w_{\Omega_1})}^2$$
and sum up as in Lemma~\ref{L2 orthogonality}

Each function $Ef_{j,\theta_j'}$ is fourier supported in $\theta_j'$, in particular, fourier supported in a cylinder of radius $\frac{1}{\lambda N_1}$, height $\frac{1}{N_1^2}$. $\Omega_1$ is tiled by cylinders of radius $\lambda N_1$, height $N_1^2$ in $t$--direction.
The proof of Lemma~\ref{L2 orthogonality} works the same, 
\begin{align*}
\|Ef_{j,\theta_j}\|_{L^2_{avg}(\Omega_1)}^2&\lesssim \sum_{\theta_j'\subset \theta_j, |\theta_j'|=\frac{1}{\lambda N_1}} \|Ef_{j,\theta_j'}\|_{L^2_{avg}(w_{B_R})}^2 \\
&\lesssim \sum_{\theta_j'\subset \theta_j, |\theta_j'|=\frac{1}{\lambda N_1}} \|Ef_{j,\theta_j'}\|_{L^4_{avg}(w_{B_R})}^2
\end{align*}

For the $L^{\infty}$--estimate, we apply Cauchy Schwartz inequality:
\begin{align*}
\|Ef_{j,\theta_j}\|_{L^{\infty}(\Omega_1)}^2&\leq (\frac{\lambda}{N_1})^{d}\sum_{\theta_j'\subset\theta_j, |\theta_{j}'|=\frac{1}{\lambda N_1}} \|Ef_{j,\theta_j'}\|_{L^{\infty}(\Omega_1)}^2\\
&\lesssim  (\frac{\lambda}{N_1})^{d}\sum_{\theta_j'\subset\theta_j, |\theta_{j}'|=\frac{1}{\lambda N_1}} \|Ef_{j,\theta_j'}\|_{L^{4}_{avg}(w_{\Omega_1})}^2
\end{align*}
{ The last inequality is an application of Lemma~\ref{L infinity}. 
Note $f_{\theta_j'}$ is supported in a ball of scale $\frac{1}{\lambda N_{1}}$, and inside a box $C$ of size $ \frac{1}{N_1^2}\times \frac{1}{\lambda N_1} \times \cdots \times \frac{1}{\lambda N_1}$. We can make a affine transform of $C$ into a cube $Q^{*}$ of scale $\lambda_{N_{1}}$, which  on the physical side would transform $\Omega_{1}$ into a cube of scale $\lambda N_{1}$ . We apply Lemma~\ref{L infinity} after the affine transformation and then transform back. (Note in those setting, cube is no different than a ball.) }

We apply H\"{o}lder's inequality to conclude the argument. 

\subsection{Proof of Lemma \ref{base case}}\label{base}
Let $\lambda\leq N_1$.  We first note that  we can use{ finitely overlapping balls $B_{\lambda N_{1}}$} to cover $\Omega$ and that  $N_{2}\lesssim 1$.
Applying  Corollary \ref{Tcor} with $\delta= \frac{1}{\lambda N_1}$ and $\upsilon = \frac{N_2}{N_1}$ we have  
\begin{align*}
\int |Ef_1 Ef_2|^2w_{B_{\lambda {N_{1}}}}&\lesssim (\lambda N_2)^{d-1}\Big| \frac{\log \lambda + \log N_1}{\log N_1 - \log N_2}\Big|^C  \sum_{|\theta_j|= \frac{1}{\lambda N_1}} \int |Ef_{1, \theta_1} Ef_{2,\theta_2}|^2 w_{B_{\lambda N_{1}}}\\
&\lesssim (\lambda N_2)^{d-1}N_{2}^{\epsilon}  \sum_{|\theta_j|= \frac{1}{{\lambda N_1}}} \prod_{j=1}^2\|Ef_{j,\theta_j}\|_{L^4(w_{B_{\lambda {N_{1}}}})}^2.
\end{align*}
With parallel decoupling, Lemma \ref{parallel}, then the  desired estimate follows.  (As remarked in Remark \ref{rrr}, one can assume $N_{1}\geq 2N_{2}$.)

\subsection{Proof of Lemma \ref{InductionN2} }\label{subsectionani}
Let $\lambda \leq N_{1}$. 

We have the following two cases: 
	\begin{itemize}
		\item Case 1: $d=2$, $N_1\geq \lambda \geq \frac{N_1}{N_2}$,  and $N_{2}' = \Big(\frac{N_1}{\lambda}\Big)$,

		\item Case 2: $d\geq 3$, $N_{1}\geq \lambda \geq \frac{N_{1}}{N_{2}^{2}}$, and $N_{2}'=(\frac{N_{1}}{\lambda})^{1/2}.$
	\end{itemize}

 It is easy to check that we only need to show that 
	\begin{equation}\label{k3}
		 K(\lambda, N_1, N_2) \lesssim K(\lambda, N_1, N_2')\big(\frac{N_1}{N_2'}\frac{N_2}{N_1}\big)^{\frac{d-1}{2}}.
	\end{equation}
We claim that 
	\begin{equation}\label{dec1}
		\|Ef_{1}Ef_{2}\|_{L^{4}_{avg}(w_{B_{N_{1}^{2}}})}\lesssim (\frac{N_{2}/N_{1}}{N'_{2}/N_{1}})^{d-1/2}\prod_{j=1}^2\left(\sum_{|\theta|=\frac{N'_{2}}{ N_1}}\|Ef_{j,\theta}\|_{L^4_{avg}(w_{B_{N_{1}^{2}}})}^2\right)^{1/2}.
	\end{equation}
Since $\lambda\leq N_{1}$, we cover $B_{N_{1}^{2}}$with balls of radius $\lambda N_{1}$. Thus by parallel decoupling, to prove \eqref{dec1},  we only need to show 
	\begin{equation}\label{dec2}
		\|Ef_{1}Ef_{2}\|_{L^{4}_{avg}(w_{B_{\lambda N_{1}}})}\lesssim (\frac{N_{2}/N_{1}}{N_{2}'/N_{1}})^{d-1/2}\prod_{j=1}^2\left(\sum_{|\theta|=\frac{1}{\lambda N_1}}\|Ef_{j,\theta}\|_{L^4_{avg}(w_{B_{\lambda N_{1}}})}^2\right)^{1/2}.
	\end{equation}
Note  that since $\lambda N_{1}\geq \frac{1}{N'_{2}/N_{1}}$, estimate  \eqref{dec2} follows from Corollary \eqref{Tcor} by setting $\delta=N_{2}/N_{1}, \upsilon=N'_{2}/N_{1}$ via  interpolation and local constant arguments as in Section \ref{ss1}.

By the definition of 
$K(\lambda, N_{1},N_{2}) $,  we have  that for any $\theta_{1},\theta_{2}$ in \eqref{dec2}, 
	\begin{equation}\label{dec3}
		\|Ef_{1,\theta_{1}}Ef_{2,\theta_{2}}\|_{L^4_{avg}(w_{B_{\lambda N_{1}}})}\lesssim  \lambda^{d/2} K(\lambda, N_1, N'_2) \prod_{j=1}^2\left(\sum_{|\theta'_{j}|=\frac{1
 }{\lambda N_1}, \theta_{j}\subset \theta_{j}}\|Ef_{j,\theta'_{j}}\|_{L^4_{avg}(w_{\Omega})}^2\right)^{1/2}.
	\end{equation}
Plugging \eqref{dec3} into \eqref{dec2}, clearly \eqref{k3} follows.

\section{Induction procedure and proof of Lemma \ref{criticallll} }\label{sectionind}
To conclude the proof of Proposition \ref{transversal}, we are left with the proof of  Lemma \ref{criticallll}. For this lemma the proof relies on induction on $N_{2}$. The base case $N_{2}\lesssim 1$ is resolved by Lemma \ref{base case}, and by Lemma \ref{InductionN2}, so we need only to induct until $\lambda=\frac{(N_{2})^{d-1}}{N_{1}}$. 

Let $f_{1}, f_{2}$ be as in Lemma \ref{criticallll}. Applying  Lemma \ref{projection to 2D}, taking $v=N_{1}/N_{2}$ and  $R=N^{2}_{1}$, we could decouple the $\frac{N_2}{N_1}$ caps into $(\frac{N_2}{N_1},\frac{N_2^2}{N_1^2})$ plates without any loss, i.e. 
\begin{equation}\label{preinduction}
\int |Ef_1 Ef_2|^2 w_{B_{N_{1}^{2}}}  \lesssim \sum_{\tau_1, \tau_2}\int |Ef_{1,\tau_1} Ef_{2,\tau_2}|^2 w_{B_{N_{1}}^{2}}. 
\end{equation}
Here $\tau_{i}$ are plates as described in Lemma \ref{projection to 2D}. We focus on the case when $d=2$ in $\RR^3$, the high dimensional case would be explained in the end. When $d=2$, the underlying plates become strips.  We start with  some preparation before the induction.

\subsection{Preliminary preparation for the induction}
We fix a  pair {of $(\frac{N_2}{N_1},\frac{N_2^2}{N_1^2})$ strips} $\tau_1, \tau_2$ from estimate \eqref{preinduction}. We decompose $\tau_j$ into a union of $\frac{N_2}{KN_1}\times \frac{N_2^2}{N_1^2}$ strips $\{s_j\}$.

Using the notation {\it nonadj}   short for {\it nonadjacent}, and {\it adj}  short for {\it adjacent}, we   have 
\begin{align*}
|Ef_{\tau_j}|^2&=  \sum_{s_j}|Ef_{s_j}|^2+ \sum_{s_j, s_j' adj} |Ef_{s_j} Ef_{s_j'}| + \sum_{s_j, s_j' non adj}|Ef_{s_j}Ef_{s_j'}|\\
&\leq 10 \sum_{s_j}|Ef_{s_j}|^2 +\sum_{s_j, s_j' non adj}|Ef_{s_j}Ef_{s_j'}|\\
&= I_{j,1}+ I_{j,2}
\end{align*}

\begin{align}
\int |Ef_{\tau_1}Ef_{\tau_2}|^2 w_{B_{N_{1}^{2}}} &\leq \int |(Ef_{\tau_1}^2 -I_{1,1})(|Ef_{\tau_2}^2-I_{2,1})| +Ef_{\tau_1}^2 I_{2,1} + Ef_{\tau_2}^2 I_{1,1}+ I_{1,1} I_{2,1}w_{B_{N_{1}^{2}}}\\
&\lesssim \sum_{s_j, s_j' nonadj} \int |Ef_{s_1}Ef_{s_1'}Ef_{s_2}Ef_{s_2'}| w_{B_{N_{1}^{2}}} +\sum_{s_1, s_2}\int |Ef_{s_1}Ef_{s_2}|^2 w_{B_{N_{1}^{2}}}\label{firred}
\end{align}

The last inequality follows from Lemma~\ref{L2 orthogonality} and Lemma~\ref{reduce}.

 The reason why we want to have non-adjacent parts is that we would like transversality (after rescaling) on the other direction.  Formula \eqref{firred} will  the starting point of  our induction.

\medskip
 For the second term in \eqref{firred}, we will later directly use  induction ( not relying  on parallel rescaling) on $N_{2}$ and reduce everything to the known base   case $N_2 =1$.
 
 \medskip
 For the first term,  using Cauchy-Schwartz  
 \begin{equation}
\int |Ef_{s_1}Ef_{s_1'}Ef_{s_2}Ef_{s_2'}| w_{B_{N_{1}^{2}}} \leq \left(\int |Ef_{s_1}Ef_{s_1'}|^2 w_{B_{N_{1}^{2}}}\right)^{1/2}\left(\int |Ef_{s_2}Ef_{s_2'}|^2w_{B_{N_{1}^{2}}}\right)^{1/2}.
\end{equation}

We point out here that in what follows   we do not rely on the  bilinear transversality between $s_{1}$ and $s_{2}$ (or $s_{1}$ and $s_{2}'$), which is already handled in Lemma \ref{projection to 2D}. Instead  we will rely on the bilinear transversality between  $s_{1}$ and $s_{1}'$, (or $s_{2}, s_{2}'$), since they are not adjacent. This  transversality is most clear when one applies parabolic rescaling.

\medskip
Let us now turn to the term  $\int |Ef_{s_2}Ef_{s_2'}|^2 w_{\Omega}$, when $s_{2}, s_{2}'$ are non adjacent. The term with $s_{1},  s'_{1}$ is  handled similarly, though one may need to rotate the coordinates. 

\medskip
Finally we point out here   that $K$ would be chosen large later and   any (fixed) power of $K$ will  not impact the final estimate. In particular, in the following estimates we would not worry about losing  powers of $K$.

\medskip

Without loss of generality, we assume 
\begin{itemize}
\item $s_{2}$ is the strip that $\{(a_{1},a_{2},a_{1}^{2}+a_{2}^{2}) | \, |a_{1}|\leq  N^{2}_{2}/N^{2}_{1}, |a_{2}|\leq {N_{2}/KN_{1}}\}$

\smallskip
\item $s_{2}'$ is the strip that $\{(b_{1},b_{2}, b_{1}^{2}+b_{2}^{2})|\, |b_{1}|\leq N_{2}^{2}/N^{2}_{1}, |b_{2}-CN_{2}/KN_{1}|\leq N_{2}/KN_{1}\}, C\geq 10$. (Here $10$ is of course just some universal constant.)
\end{itemize}

\subsection{Parabolic rescaling}
The next step, parabolic scaling, is standard in decoupling types results; we give the details here for the convenience of the  reader. 

\medskip
 Note $s_{2}, s_{2}'$ lie on the same $\frac{N_2}{N_1}$ cap. We rescale the $\frac{N_2}{N_1}$ cap to radius 1. 
 By a slight abuse of notation, we regard $f_{s_{i}}$ as a function depending only on two variables $(\xi_{i,1}, \xi_{i,2})$.
For convenience notation, we let $h_{1}=f_{s_{2}}, \, h_{2}=f_{s_{2}'}$.
Let also $g_{i}(\eta_{i,1},\eta_{i,2}): =h_{i}((N_{2}/N_{1})\eta_{i,1}, (N_{2}/N_{1})\eta_{i,2})$.
 
Note now
\begin{itemize}
\item $g_{1}$ is supported in the strip of $\{(a_{1},a_{2},a_{1}^{2}+a_{2}^{2}) \, |  \, |a_{1}|\leq  N_{2}/N_{1}, |a_{2}|\leq 1/K\}$
\item $g_{2}$ is supported in the strip of $\{(b_{1},b_{2}, b_{1}^{2}+b_{2}^{2}) \, | \, |b_{1}|\leq N_{2}/N_{1}, |b_{2}-C/K|\leq 1/K\}, C\geq 10$
\end{itemize}
 Note  $g_1$, $g_2$ are supported on a pair of transverse $\frac{N_2}{N_1}\times 1$ strips\footnote{ Strictly speaking, we need them to support on a pair of $ \frac{N_{2}}{N_{1}}\times \frac{1}{100}$ strips, we neglect this technical point here.} due to the non adjacency  of $s_{2},  s_{2}'$.
 We point out here  the transversality between  $g_{1},  g_{2}$ is  not as in the assumption of    Lemma \ref{projection to 2D}, but  it is in the sense of Remark \ref{Ktran}, which usually cause a loss of $K$ in the estimate, but this  does not matter.
  
  The {\it parabolic scaling}  says the following:
  \begin{claim}\label{ps}
  Let  $Eg_{i}(y_{1},y_{2},y_{3})=Eh_{i}(N_{1}/N_{2}y_{1},N_{1}/N_{2} y_{2}, N_{1}^{2}/N_{2}^{2} y_{3})$, and let $D$ be domain in $\RRR^{3}$ and let 
  $$\tilde{D}:=\{(y_{1},y_{2},y_{3}): N_{1}/N_{2}y_{1},N_{1}/N_{2} y_{2}, N_{1}^{2}/N_{2}^{2} y_{3}\in D\},$$ then it follows from standard change of variables technique that the following two estimates, with the  same constant $A$, are equivalent:
  \begin{equation}\label{before}  
 \|Eh_{1} Eh_{2}\|_{L^2_{avg}(w_{D})} \lesssim A\prod_{j=1}^2\left(\sum_{|\theta|=\frac{1}{\lambda N_1}}\|Ef_{s_j,\theta}\|_{L^4_{avg}(w_{D})}^2\right)^{1/2},
  \end{equation}
  \begin{equation}\label{after}
 \|Eg_{1} Eg_{2}\|_{L^2_{avg}(w_{\tilde{D}})} \lesssim A\prod_{j=1}^{2} \left(\sum_{|\tilde{\theta}|=\frac{1}{\lambda N_2}}\|Eg_{j,\tilde{\theta}}\|_{L^4_{avg}(w_{\tilde{D}})}^2\right)^{1/2}.
  \end{equation}
  \end{claim}
we then concentrate on \eqref{after}.

Take $D=B_{N_{1}^{2}}$, then $\tilde{D}=[0,N_{2}^{2}]\times [0,N_{1}N_{2}]^{2}$.  (Here, without loss of generality, we regard  $B_{N_{1}^{2}}$ as $[0,N_{1}^{2}]^{3}$.) For convenience of notation, we set $\tilde{\Omega}=[0, N_{2}^{2}]\times [0,N_{1}N_{2}]^{2}$. The parabolic rescaling gives 
\begin{lem}\label{reduction1}
Assume  $g_{1},  g_{2}$ are  two general functions defined on the parabola.  Let $g_{1}$ be supported in  a strip of size $N_{2}/N_{1}\times 1$ around $(0,0,0)$, and $g_{2}$ be supported in  a strip of size $N_{2}/N_{1}\times 1$  around $(0,1,1)$. 
If for some constant $A$, one has (for all such $g_{1}, g_{2}$),
\begin{equation}\label{gd}
 \|Eg_{1} Eg_{2}\|_{L^2_{avg}(w_{\tilde{\Omega}})} \lesssim A \left(\sum_{|\tilde{\theta}|=\frac{1}{\lambda N_2}}\|Eg_{j,\tilde{\theta}}\|_{L^4_{avg}(w_{\hat{\Omega}})}^2\right)^{1/2},
\end{equation}
then for the same constant $A$, one has 
\begin{equation}\label{sd}
 \|Ef_{s_2} Ef_{s'_2}\|_{L^2_{avg}(w_{B_{N_{1}^{2}}})} \lesssim K^{C}A\left(\sum_{|\theta|=\frac{1}{\lambda N_1}}\|Ef_{s_2,\theta}\|_{L^4_{avg}(w_{B_{N_{1}^{2}}})}^2\right)^{1/2}\left(\sum_{|\theta|=\frac{1}{\lambda N_1}}\|Ef_{s'_2,\theta}\|_{L^4_{avg}(w_{B_{N_{1}^{2}}})}^2\right)^{1/2}.
\end{equation}
\end{lem}
\begin{rem}
After rescaling, the relevant $g_{1}, g_{2}$ should be supported around $(0,0,0)$ and $(0, 1/K, 1/K^{2})$ rather than $(0,0,0)$ and $(0,0,1)$. We state our lemma for $g_{1}, g_{2}$ supported around  $(0,0,0)$ and $(0,1,1)$ to be consistent with the statement in Lemma \ref{projection to 2D}. This causes a loss of $K^{C}$, but   we emphasize again that any loss due to a power of $K$ would be irrelevant in the proof.
\end{rem}

We end this section by introducing some notation.

\medskip

Let $g_{1}, g_{2}$ be as in Lemma \ref{reduction1}, we define $A(\lambda, N_{1},N_{2})$ to be the best constant such that
\begin{equation}\label{notegd}
 \|Eg_{1} Eg_{2}\|_{L^2_{avg}(w_{\tilde{\Omega}})} \lesssim A(\lambda, N_{1}, N_{2}) \left(\sum_{|\tilde{\theta}|=\frac{1}{\lambda N_2}}\|Eg_{j,\tilde{\theta}}\|_{L^4_{avg}(w_{\hat{\Omega}})}^2\right)^{1/2}.
\end{equation}
Then we can restate Lemma \eqref{reduction1} .
\begin{lem}\label{tansferlemma}
For $j=1,2$,  we have 
\begin{equation}\label{ssd}
\begin{aligned}
& \|Ef_{s_j} Ef_{s'_j}\|_{L^2_{avg}(w_{B_{N_{1}^{2}}})}\\
  \lesssim 
 &K^{C}A(\lambda, N_{1}, N_{2})\left(\sum_{|\theta|=\frac{1}{\lambda N_1}}\|Ef_{s_j,\theta}\|_{L^4_{avg}(w_{B_{N_{1}^{2}}})}^2\right)^{1/2}\left(\sum_{|\theta|=\frac{1}{\lambda N_1}}\|Ef_{s'_j,\theta}\|_{L^4_{avg}(w_{B_{N_{1}^{2}}})}^2\right)^{1/2}.
 \end{aligned}
\end{equation}
\end{lem}

\subsection{The induction procedure}\label{sectioninductionpro}
\subsubsection{Before induction}
Now we are ready to start the induction for the proof of Lemma \ref{criticallll}.  We emphasize here the induction is on $N_{2}$, (though mixed with induction on $K$). Note we are now in dimension $d=2$.

We need to show that  for all $1\leq N_{2}\leq N_{1}$ and $\lambda \leq N_{1}$,  one  has $$K (\lambda,N_{1},N_{2} )\lesssim N_{2}^{\epsilon}(\frac{1}{\lambda}+\frac{N_{2}}{N_{1}})^{1/2}.$$
Note the base case $N_{2}=1$ is already established in Corollary \ref{base case}.  And with Lemma \ref{InductionN2}, we need only to perform induction  until $\lambda=N_{2}/N_{1}$. 
 
 We will work on $A(\lambda, N_{1}, N_{2})$ defined in \eqref{notegd} to explore the transversality between nonadjacent strips.  The induction process is two fold in some sense. We will induct on $N_{2}$ to better understand $K(\lambda, N_{1}, N_{2})$, and in turn we find more information about $A(\lambda, N_{1}, N_{2})$, which in turn gives a better understanding of $K(\lambda ,N_{1}, N_{2})$.

This is a final   summary before we start the induction. Recall, we have  \eqref{preinduction} and  \eqref{firred}, thus we have
\begin{equation}\label{finalinduce}
 \begin{aligned}
               &\int |Ef_1 Ef_2|^2 w_{B_{N_{1}^{2}}}\\
 \lesssim  &\int \sum_{s_{j}, s_{j}^{' }nonadj}\int |Ef_{s_{1}}Ef_{s_{1}'}Ef_{s_{2}}Ef_{s_{2}'}|w_{B_{N_{1}^{2}}}\\
 +& \int_{s_{1}, s_{2}}\int |Ef_{s_{1}}Ef_{s_{2}}|w_{ B_{N_{1}^{2}}}.
 \end{aligned}
 \end{equation}
 Also recall  that $s_{1}, s_{1}' ,s_{2}, s_{2}'$ are all  $(N_{2}/N_{1})^{2}\times N_{2}/KN_{1}$ strips.
 The second term can be  easily handled by direct induction, (which is not the main point of the induction procedure explained later). Indeed, if there were only the  second term in  \eqref{finalinduce}, since $s_{1}, s_{2}$ are both contained in caps of radius $(N_{2}/KN_{1})$, then \eqref{finalinduce} already reduces the decoupling problem for  $f_{i}$ supported in caps of size $N_{2}/N_{1}$ into the decoupling problem for $f_{i}$ supported in caps of size $N_{2}/KN_{1}$, which reduce $N_{2}$ to $N_{2}/K$.

We will focus on the first term of \eqref{finalinduce}. H\"older inequality gives
\begin{equation}\label{baseindcution}
\int |Ef_{s_1}Ef_{s_{1}'}Ef_{s_{2}}Ef_{s_{2}'}|w_{B_{N_{1}^{2}}}\leq \prod_{j=1}^{2}\left(\int  |Ef_{s_{j}}Ef_{s_{j}'}|^2w_{B_{N_{1}^{2}}}\right)^{1/2}.
\end{equation}
Estimate \eqref{baseindcution} is the start point of the analysis in the following Subsections.

We summarize in the lemma below how \eqref{baseindcution} and  \eqref{finalinduce} come together to highlight the relevance of $A(N_{1}, N_{2}, \lambda)$ in the induction procedure.
\begin{lem}\label{finalinductionlemma}
When $\lambda \leq N_{1}/N_{2}$ and $\lambda \leq N_{1}$, we have
\begin{equation}\label{fikey}
K(N_{1}, N_{2}, \lambda)\lesssim K^{C}\frac{1}{\lambda}A(N_{1}, N_{2},\lambda)+K(N_{1}, N_{2}/K, \lambda)
\end{equation}
\end{lem}
Note that the assumption of Lemma \ref{finalinductionlemma} always holds during the induction procedure to prove Lemma \ref{criticallll}.
\begin{proof}[Proof of Lemma \ref{finalinductionlemma}]
 Applying Lemma \ref{tansferlemma}, we have
 \begin{equation}\label{ssd2}
\begin{aligned}
& \|Ef_{s_j} Ef_{s'_j}\|_{L^2_{avg}(w_{B_{N_{1}^{2}}})}\\
  \lesssim 
 &K^{C}A(N_{1}, N_{2}, \lambda)\left(\sum_{|\theta|=\frac{1}{\lambda N_1}}\|Ef_{s_j,\theta}\|_{L^4_{avg}(w_{B_{N_{1}^{2}}})}^2\right)^{1/2}\left(\sum_{|\theta|=\frac{1}{\lambda N_1}}\|Ef_{s'_j,\theta}\|_{L^4_{avg}(w_{B_{N_{1}^{2}}})}^2\right)^{1/2}.
 \end{aligned}
\end{equation}

Plugging \eqref{ssd2} into  \eqref{baseindcution}, and then plugging into \eqref{finalinduce}, we derive
\begin{equation}\label{finalinduce2}
 \begin{aligned}
               &\|Ef_1 Ef_2|\|_{l^{2}(w_{B_{N_{1}^{2}}})}\\
 \lesssim  &K^{C}\lambda \left(\frac{1}{\lambda}\right)^{1/2}\prod_{i=1}^{2}\left(\sum_{|\theta|=\frac{1}{\lambda N_{1}}}
 \|Ef_{i,\theta}\|_{L^4_{avg}(w_{B_{N_{1}^{2}}})}^{2}\right)^{1/2}\\
 +& \left(\sum_{|\theta|=\frac{N_{2}}{\lambda KN_{1}}}
 \|Ef_{i,\theta}\|_{L^4_{avg}(w_{B_{N_{1}^{2}}})}^{2}\right)^{1/2}.
 \end{aligned}
 \end{equation}

 Thus we derive 
 \begin{equation}
 \lambda K(N_{1},N_{2},\lambda )\lesssim  K^{C}A(N_{1}, N_{2}, \lambda)+\lambda K(N_{1}, N_{2}/K, \lambda)
 \end{equation}
 Thus, Lemma \ref{finalinductionlemma} follows.

\end{proof}
Now we are ready to start with  the induction procedure on $N_{2}$. We emphasize again that by Lemma \ref{InductionN2}
we only need to consider the case  $\lambda \leq N_{1}/N_{2}$.

\subsubsection{First induction: Case $N_{2}^{2}\leq N_{1}$}\label{firstinductionsection}
It will become clear in the following  proof why we  choose the first splitting point at $N_{1}=N_{2}^{2}$.  We start with an estimate for $A(\lambda, N_{1}, N_{2})$.
We have 
\begin{lem}\label{lemmaA}
When $N_{2}\leq N_{1}^{2}$,  $\lambda \leq N_{1}$, $\lambda \leq N_{1}/N_{2}$,
\begin{equation}
 A(\lambda, N_{1}, N_{2})\lesssim \lambda^{1/2}\equiv \lambda \lambda^{-1/2}.
 \end{equation}
\end{lem} 
Assuming Lemma \ref{lemmaA} for the moment, let us finish the proof of Lemma \ref{criticallll}  when  $N_{1}\geq N_{2}^{2}$.
Applying Lemma  \ref{lemmaA} with Lemma  \ref{finalinductionlemma}, we derive
\begin{equation}
 K(N_{1},N_{2},\lambda )\lesssim K^{C}\lambda \left(\frac{1}{\lambda}\right)^{1/2}+K(N_{1}, N_{2}/K, \lambda)
 \end{equation}
 when $N_{1}\geq N_{2}^{2}$  and $\lambda\leq N_{1}/N_{2}$.
 Choosing $1\ll K\sim N_{2}^{\epsilon^{10}}$, performing  induction on $N_{2}$ again, and recalling that the case  $N_{2}\lesssim 1$ is covered by Lemma \ref{base case}, then Lemma \ref{criticallll} follows when $N_{1}\geq N^{2}_{2}$.

\medskip
Now, we turn to the proof of Lemma \ref{lemmaA}.
 \begin{proof}[Proof of Lemma \ref{lemmaA}]
 Since $N_{1}\leq N_{2}^{2}$, thus $\frac{N_{2}}{N_{1}}\leq \frac{1}{N_{2}}$.   (It is exactly because of this  that we decided our first splitting point $N_{1}\leq N_{2}^{2}$). Thus, the support of $g_{1}, g_{2}$ appearing  in \eqref{notegd} are (contained in) strips of size $\frac{1}{N_{2}}\times 1$. Thus, in a ball of radius $N_{2}^{2}$,  we have 
 \begin{equation}\label{projectto2dagain}
 \int |Eg_{1}Eg_{2}|w_{B_{N_{2}^{2}}}\lesssim \sum_{|\theta_{i}|=\frac{1}{N_{2}}, \theta_{i}\subset supp \,  g_{i}} \int |E_{g_{1},\theta_{1}}Eg_{2,\theta_{2}}|w_{B_{N_{2}^{2}}}.
 \end{equation}
 The proof of \eqref{projectto2dagain} is essentially the same as the proof of Lemma \ref{projection to 2D} and we leave it to reader.
 
 Note one can use balls $B_{N_{2}^{2}}$ to cover $\tilde{\Omega}:=[0,N_{2}^{2}]\times [0,N_{1}N_{2}]^{2}$, (since $N_{1}\geq N_{2}$) thus we extend \eqref{projectto2dagain} to
 \begin{equation}\label{ddd1}
  \int |Eg_{1}Eg_{2}|w_{{\tilde{\Omega}}}\lesssim \sum_{|\theta_{i}|=\frac{1}{N_{2}}, \theta_{i}\subset supp \, g_{i}} \int |E_{g_{1},\theta_{1}}Eg_{2,\theta_{2}}|w_{{\tilde{\Omega}}}.
 \end{equation}
 
We claim for any fixed $\theta_{1}, \theta_{2}$, one has 
\begin{equation}\label{ddd4}
 \|Eg_{1,\theta_{1}}Eg_{2, \theta_{2}}\|_{L^{2}(w_{\tilde{\Omega}})}\lesssim  \lambda \lambda^{-1/2}\prod_{i=1}^{2}
 \left(\sum_{\tilde{\theta_{i}}\subset \theta_{i}, \|\tilde{\theta_{i}}=\frac{1}{\lambda N_{2}}\|}\|Eg_{i, \tilde{\theta_{i}}}\|_{L^{4}(w_{\tilde{\Omega}})}\right)^{1/2}
\end{equation} 
Plugging \eqref{ddd4} into \eqref{ddd1}, we have
\begin{equation}
A(N_{1},N_{2}, \lambda)\lesssim \lambda (\frac{1}{\lambda})^{1/2}.
\end{equation}  
 and the Lemma \ref{lemmaA} follows.
 
 Now we are left with the proof of \eqref{ddd4}. Let $N_{1}'=N_{2}, N_{2}'=N_{2}^{2}/N_{1}\lesssim 1$.
When $ N_{1}'=N_{2}\leq \lambda$, recall the definition of  $\tilde{K}(\lambda ,N_{1}, N_{2})$ in \eqref{best3} and apply Lemma \ref{big lambda}, we have 
\begin{equation}\label{ddd3}
\begin{aligned}
K(N_{1}', N_{2}', \lambda)\lesssim (N'_{2})^{\epsilon}\left(\frac{N'_{2}}{N_{1}'}\right)^{1/2}\lesssim \left(\frac{1}{\lambda}+\frac{N_{2}}{N_{1}}\right)^{1/2}\lesssim \lambda^{-1/2}.
\end{aligned}
\end{equation} 
The last inequality in \eqref{ddd3} follows because we always have $\lambda\leq N_{1}/N_{2}$ in the whole induction process.
Note \eqref{ddd3} implies
\begin{eqnarray} \label{ddd2}
&& \|Eg_{1,\theta_{1}}Eg_{2, \theta_{2}}\|_{L^{2}(w_{[0,N_{2}^{2}]\times [0,\lambda N_{2}]^{2}})}\\\notag
&\lesssim&  \lambda \tilde{K}(N_{1}' ,N_{2}', \lambda) \prod_{i=1}^{2}
 \left(\sum_{\tilde{\theta_{i}}\subset \theta_{i}, \|\tilde{\theta_{i}}=\frac{1}{\lambda N_{2}}\|}\|Eg_{i, \tilde{\theta_{i}}}\|_{L^{4}(w_{[0,N_{2}^{2}]\times [0,\lambda N_{2}]^{2}})}\right)^{1/2}.
 \end{eqnarray}

Since $\lambda\leq N_{1}$, (which is also always the case during the induction process ), $\tilde{\Omega}$ can be covered by the translations of $[0,N_{2}^{2}]\times [0,\lambda N_{2}]$, thus \eqref{ddd2} implies \eqref{ddd4} by the parallel decoupling Lemma \ref{parallel}.

When $\lambda\leq N_{1}'$, since $N'_{2}\lesssim 1$, by Lemma \ref{base case}, we have 
\begin{equation}
K(\lambda, N'_{1}, N_{2}')\lesssim \lambda^{-1/2}.
\end{equation}
Thus,
\begin{equation} \label{ddd5}
 \|Eg_{1,\theta_{1}}Eg_{2, \theta_{2}}\|_{L^{2}(w_{B_{N_{2}^{2}}})}\lesssim  \lambda \lambda^{-1/2} \prod_{i=1}^{2}
 \left(\sum_{\tilde{\theta_{i}}\subset \theta_{i}, \|\tilde{\theta_{i}}=\frac{1}{\lambda N_{2}}\|}\|Eg_{i, \tilde{\theta_{i}}}\|_{L^{4}(w_{B_{N_{2}^{2}}})}\right)^{1/2}.
 \end{equation}
 Since one can use $B_{N_{2}^{2}}$ and its translations to cover $\tilde{\Omega}$, \eqref{ddd5} implies \eqref{ddd4} by the parallel decoupling Lemma \ref{parallel}.

 \end{proof}

\subsubsection{Second induction: Case $N_2^{3/2} \leq N_1\leq N_2^2$ }\label{secondind}
\begin{lem}\label{LemmaA2}
When $N_{2}^{3/2}\leq N_{1}\leq N_{2}^{2}$,  $\lambda \leq N_{1}$ and  $\lambda \leq N_{1}/N_{2}$, we have 
\begin{equation}\label{criticalestimate}
 A(\lambda, N_{1}, N_{2})\lesssim \lambda^{1/2}\equiv \lambda \lambda^{-1/2}.
 \end{equation}
\end{lem}
Clearly, using  Lemma \ref{finalinductionlemma} and arguing  as   in Section \ref{firstinductionsection}, Lemma \ref{criticallll} follows from Lemma \ref{LemmaA2} when $N_{2}^{3/2}\leq N_{1}\leq N_{2}^{2}$.

Now we are left with proof of Lemma \ref{LemmaA2},  i.e.  the estimate \eqref{criticalestimate}.
We will prove  that estimate \eqref{criticalestimate}, in case $N_{2}^{3/2}\leq N_{1}\leq N_{2}^{2}$,  follows from the fact that Lemma \ref{criticallll} holds when $N_{2}^{2}\geq N_{1} $, (given Lemma \ref{big lambda}).

\begin{proof}[Proof of Lemma \ref{LemmaA2}]
The proof starts similarly as the proof of Lemma \ref{lemmaA}, note now we have $N_{2}/N_{1}\geq 1/N_{2}$. As we derived \eqref{projectto2dagain}, we have in a ball of radius $N_{1}^{2}/N^{2}_{2}$,
\begin{equation}
 \int |Eg_{1}Eg_{2}|w_{B_{(\frac{N_{1}}{N_{2}})^{2}}}\lesssim \sum_{|\theta_{i}|=\frac{N_{2}}{N_{1}}, \theta_{i}\subset supp \, g_{i}} \int |E_{g_{1},\theta_{1}}Eg_{2,\theta_{2}}|w_{B_{(N_{1}/N_{2})^{2}}}.
\end{equation}
Note one can use $B_{(\frac{N_{1}}{N_{2}})^{2}}$ and its translations to cover $\tilde{\Omega}$, thus we have
\begin{equation}\label{eee1}
  \int |Eg_{1}Eg_{2}|w_{{\tilde{\Omega}}}\lesssim \sum_{|\theta_{i}|=\frac{N_{2}}{N_{1}}, \theta_{i}\subset supp \, g_{i}} \int |E_{g_{1},\theta_{1}}Eg_{2,\theta_{2}}|w_{{\tilde{\Omega}}}.
 \end{equation}
 
 The following procedure is essentially the same as in the first induction. Note that to prove \eqref{criticalestimate} we only need to further show that for fix $\theta_{1}, \theta_{2}$,
 \begin{equation}\label{eee2}
 \|Eg_{1,\theta_{1}}Eg_{2, \theta_{2}}\|_{L^{2}(w_{\tilde{\Omega}})}\lesssim  \lambda \lambda^{-1/2}\prod_{i=1}^{2}
 \left(\sum_{\tilde{\theta_{i}}\subset \theta_{i}, |\tilde{\theta_{i}}|=\frac{1}{\lambda N_{2}}}\|Eg_{i, \tilde{\theta_{i}}}\|_{L^{4}(w_{\tilde{\Omega}})}\right)^{1/2},
\end{equation} 
where  now $|\theta_{i}|=\frac{N_{2}}{N_{1}}$.

Let $N_{1}'=N_{2}$, $N_{2}'=\frac{N_{2}^{2}}{N_{1}}$, note we have $N_{1}'\geq (N'_{2})^{2}$ since $N_{1}\geq N_{2}^{3/2}$. When $\lambda \geq N'_{1}$,
we have by Lemma \ref{big lambda} 
\begin{equation}\label{eee3}
\begin{aligned}
 \|Eg_{1,\theta_{1}}Eg_{2, \theta_{2}}\|_{L^{2}(w_{[0,N_{2}^{2}]\times [0,\lambda N_{2}]^{2}})}\lesssim  \lambda \left(\frac{N_{2}'}{N_{1}'}\right)^{-1/2}\prod_{i=1}^{2}
 \left(\sum_{\tilde{\theta_{i}}\subset \theta_{i}, |\tilde{\theta_{i}}|=\frac{1}{\lambda N_{2}}}\|Eg_{i, \tilde{\theta_{i}}}\|_{L^{4}(w_{[0,N_{2}]\times [0,\lambda N_{2}]^{2}})}\right)^{1/2}.
 \end{aligned}
\end{equation}
Since one can use $[0,N_{2}^{2}]\times [0,\lambda N_{2}]^{2}$ to cover $\tilde{\Omega}$, \eqref{eee2} follows from \eqref{eee3},
(note $\frac{N_{2}'}{N_{1}'}=\frac{N_{2}}{N_{1}}\leq \lambda^{-1}$).

When $\lambda\leq N_{1}'$, since one can use $B_{N_{2}^{2}}$ to cover  $\tilde{\Omega}$, to prove \eqref{eee2}, we need only to show
\begin{equation}\label{eee4}
\begin{aligned}
 \|Eg_{1,\theta_{1}}Eg_{2, \theta_{2}}\|_{L^{2}(w_{B_{N_{2}^{2}}})}\lesssim  \lambda \lambda^{-1/2}\prod_{i=1}^{2}
 \left(\sum_{\tilde{\theta_{i}}\subset \theta_{i}, \|\tilde{\theta_{i}}=\frac{1}{\lambda N_{2}}\|}\|Eg_{i, \tilde{\theta_{i}}}\|_{L^{4}(w_{B_{N_{2}^{2}}})}\right)^{1/2},
 \end{aligned}
\end{equation}
which is equivalent to $K(N_{1}', N_{2}', \lambda)\leq \frac{1}{\lambda}$. But recall that $ N_{1}'\geq (N'_{2})^{2}$, thus this is exactly what we proved in first induction, i.e. Lemma \ref{criticallll} holds when $N_{1}\geq N_{2}^{2}$.
\end{proof}

\subsubsection{Later inductions and the conclusion of the induction process} 
Recall that the {\it first induction} covers the case $N_{1}\geq N_{2}^{2}$ and the {\it second inductions} covers the case $N_{2}^{\alpha}\leq N_{1}\leq N_{2}^{2} , \alpha=3/2$. The goal now  is to use  induction to cover the case $N_{2}^{\alpha}\leq N_{1}$, all the way to $  \alpha=1$.
The arguments here are similar to those for  the {\it second induction} presented  in Section \ref{secondind}.  Let $N'_{1}=N_{2}, N_{2}'=N_{2}^{2}/N_{1}$, then $N_{1}'\geq (N'_{2})^{\alpha}$ is equivalent to $N_{1}\geq N_{2}^{\frac{2\alpha-1}{\alpha}}$.
Once we show that  Lemma \ref{criticallll} holds when $N_{2}^{\alpha}\leq  N_{1}\leq N^{2}_{2}$, we would be able to extend Lemma \ref{LemmaA2}  to the case when $N_{2}^{\frac{2\alpha-1}{\alpha}}\leq N_{1}$, which in turn proves that  Lemma \ref{criticallll} holds when $N_{2}^{\frac{2\alpha-1}{\alpha}}\leq N_{1}\leq N_{2}^{2}$.
The induction would not end until $\alpha=1$. 
We finally point out, that only an  induction with finite steps is  involved.

To show Lemma \ref{criticallll} for a fixed  $\epsilon_{0}$, we may  pick an $\tilde{\epsilon}\ll \epsilon_{0}$, then we perform the induction for $\tilde{\epsilon}$ as above. 

After we prove Lemma \ref{criticallll} for $N_{1}\geq N_{2}^{1+\tilde{\epsilon}}$, we are left with the case $N_{1}\leq N_{2}^{1+\tilde{\epsilon}}$ .  We  first  use H\"older inequality to shrink the size of the cap from $N_{2}/N_{1}$ to $N_{2}^{1-2\tilde{\epsilon}}/N_{1}$, which only gives a loss of  $N_{2}^{C\tilde{\epsilon}}\ll N_{2}^{\epsilon_{0}}$. Then we use Lemma \eqref{criticallll} in the case $N_{1}\geq N_{2}^{1+\tilde{\epsilon}}$ again.

Thus, Lemma \ref{criticallll}  holds for all the case for our fixed $\epsilon_{0}$.

\subsection{The high dimension case}
To handle the case $d\geq 3$, we are left with the proof of Lemma \ref{base case high}. The proof is indeed similar to previous  arguments in  this section and easier. The proof relies on the linear decoupling estimate in \cite{bourgain2014proof}.

As aforementioned, applying  Lemma \ref{projection to 2D}, taking $v=N_{2}/N_{1}$ and  $R=N^{2}_{1}$, we can  decouple the $\frac{N_2}{N_1}$ caps into $(\frac{N_2}{N_1},\frac{N_2^2}{N_1^2})$ plates without any loss, i.e. \eqref{preinduction}. However, since we are in the case $\lambda\leq N_{1}/N_{2}^{2}$, indeed $N_{2}^{2}/N_{1}^{2}\lesssim \frac{1}{\lambda N_{1}}$, we only need a weaker version of \eqref{preinduction}, i.e. we only want to decouple the $\frac{N_2}{N_1}$ caps into $(\frac{N_2}{N_1},\frac{1}{\lambda N_{1}})$ plates:
\begin{equation}\label{preinduction2}
\int |Ef_1 Ef_2|^2 w_{B_{N_{1}^{2}}}  \lesssim \sum_{\tau_1, \tau_2}\int |Ef_{1,\tau_1} Ef_{2,\tau_2}|^2 w_{B_{N_{1}}^{2}} 
\end{equation}
Here $\tau_{i}$ are $(\frac{N_{2}}{N_{1}}, \frac{1}{\lambda N_1})$ plates as described in Lemma \ref{projection to 2D}. Note \eqref{preinduction2} follows from \eqref{preinduction}.

Now, for each $\tau_{i}$ fixed, we further decouple $\tau_{i}$ into $(\frac{1}{N_{1}}, \frac{1}{\lambda N_{1}})$ plates via linear decoupling in \cite{bourgain2014proof}, here recalled in \eqref{decouple}. Note direct application of linear decoupling in dimension $d$ gives us 
\begin{equation}\label{lineardecoupleaux}
\|Ef_{\tau_{i}}\|_{L^{4}(w_{B_{N_{1}^{2}}})}\lesssim N_{2}^{\epsilon}(N_{2}^{2})^{\frac{d}{4}-\frac{d+2}{8}}\left(\sum_{v_{i}\subset \tau_{i}} \|Ef_{v_{i}}\|_{L^{4}(w_{B_{N_1^2}})}^{2}\right)^{1/2},
\end{equation}  

However,   we are  able to use \eqref{decouple} when the  dimension is $d-1$ rather than $d$, because our plates are so thin (of scale $\frac{1}{\lambda N_{1}}\leq \frac{1}{N_{1}}$), which reduce the dimension by 1.
{
Indeed, Linear decoupling  \eqref{decouple} not only work for those functions which are exactly supported in parabola $P$ but also those which are supported in a $N_{1}^{-2}$ neighborhood of $P$. This is consistent in uncertainty principle, since in physical space we of scale $N_{1}^{2}$, in frequency space any scale of $N_{1}^{-2}$ cannot be differentiated. Since our plate are so thin, of scale $\frac{1}{\lambda N_{1}}\leq N^{-2}_{1}$, one  could indeed view it as a $N_{1}^{-2}$ neighborhood of some $d-1$ dimensional parabola. To be more specific and use $\tau_{2}$ as example, since $\tau_2$ is supported at the origin. Let $\pi_{t}^{-1}(\tau_2)$ be the pull back image of $\tau_2$ to the paraboloid.  The fourier inverse transform of $Ef_{\tau_2}$ is supported on $ \pi_{t}^{-1}(\tau_2)$.  One can see that if we project along $x_1$--axis, the projection image of $\pi_{t}^{-1}(\tau_2)$ is the $(\frac{1}{\lambda N_1})^2$--neighborhood of a $(d-1)$--dimensional paraboloid (a piece of length $\frac{N_2}{N_1}$).
}

Now, apply $d-1$ dimensional linear decoupling, we improve \eqref{lineardecoupleaux} into
\begin{equation}\label{lineardecouple}
\|Ef_{\tau_{i}}\|_{L^{4}(w_{B_{N_{1}^{2}}})}\lesssim N_{2}^{\epsilon}(N_{2}^{2})^{\frac{d-1}{4}-\frac{d+1}{8}}\left(\sum_{v_{i}\subset \tau_{i}} \|Ef_{v_{i}}\|_{L^{4}(w_{B_{N_1^2}})}^{2}\right)^{1/2},
\end{equation}
where  $v_{i}$ are $(\frac{1}{N_{1}}, \frac{1}{\lambda N_{1}})$ plates.

Finally, similarly to the derivation of \eqref{min},  we decouple $v_{i}$ into caps of radius $\frac{1}{\lambda N_{1}}$,
\begin{equation}\label{direct}
\|Ef_{v_{i}}\|^{4}_{L^{4}(w_{B_{N_{1}^{2}}})}\lesssim \lambda^{(d-1)} \left(\sum_{\theta_{i}\subset v_{i}}\|Ef_{\theta_{i}}\|_{L^{4}(w_{B_{N_{1}^{2}}})}^{2}\right)^{2}.
\end{equation}
We remark that   each $v_{i}$ can  be coved by $\lambda ^{d-1}$ rather than $\lambda^{d}$ caps of radius $\frac{1}{\lambda N_{1}}$.
Plugging \eqref{direct} into \eqref{lineardecouple}, then plugging it  into \eqref{preinduction2}, we derive
\begin{equation}
\|Ef_{1} Ef_{2}\|_{L^2_{avg}(w_{B_{N_{1}^{2}}})} \leq \lambda^{d-1/2}N_{2}^{\frac{d-3}{2}} \prod_{j=1}^2\left(\sum_{|\theta|=\frac{1
 }{\lambda N_1}}\|Ef_{j,\theta}\|_{L^4_{avg}(w_{B_{N_{1}^{2}}})}^2\right)^{1/2}.
\end{equation}
Thus, the desired estimate for $K(\lambda, N_{1}, N_{2})$ follows.

\appendix

\section{Sharpness of Theorem~\ref{bilinear decoupling} and Theorem~\ref{Main Theorem} }\label{sharpness}

The sharpness (up to $N_2^{\epsilon}$) of Theorem~\ref{bilinear decoupling}  is provided by the following examples.  One can also re-scale those example to show the sharpness of Theorem \ref{Main Theorem}.

We take $Ef_1 =\underset{\xi \in \Lambda_{\lambda N_1}, |\xi|\leq \frac{N_2}{N_1}}{\sum} e^{2\pi i (\xi \cdot x +|\xi|^2t)}$ and $f_2 =f_1 (\cdot - (1,0,\dots, 0))$.  Then $|Ef_1|$ is about $(\lambda N_2)^d$ at $B(0,\frac{N_1}{N_2})$ in $\RRR^{d+1}$. Note that it follows from uncertainty principle, it is locally constant in  any ball of size $\frac{N_{1}}{N_{2}}$ and one can easily compute $|Ef_{1}(0)|\sim (\lambda N_{2})^{d}$ . Also note $|Ef_1|$ has periodicity around $\lambda N_1$ in all components of $x$, (not necessarily in $t$).  The same is true for $|Ef_2|$.  Thus,
\begin{align*}
\|Ef_1Ef_2\|_{L^2(w_{\Omega})}^2 &\gtrsim (\lambda N_2)^{4d} |B(0, \frac{N_1}{N_2})| (\lambda N_1)^d\\
&\gtrsim \lambda^{5d} N_1^{2d+1} N_2^{3d-1}
\end{align*}
Each cap $\theta_j$ of radius $\frac{1}{\lambda N_1}$ contains at most one point $\xi \in \Lambda_{\lambda N_1}$. Hence $\|Ef_{j, \theta_j}\|_{L^{4}(w_{\Omega})}^4\lesssim |\Omega| = N_1^2 (\lambda N_1)^{2d}$. 
\begin{align*}
\Pi_{j=1}^2 (\sum_{|\theta_j|=\frac{1}{\lambda N_1}} \|Ef_{j,\theta_j}\|_{L^4(w_{\Omega})}^2) & \lesssim (\lambda N_2)^{2d}  N_1^2 (\lambda N_1)^{2d} \\
&\lesssim \lambda^{4d} N_1^{2d+2} N_2^{2d}
\end{align*}
This example shows that the term with $\frac{N_2^{d-1}}{N_1}$ is sharp for both $d=2$ and $d\geq 3$. 

When $d=2$, we consider the example when 
\begin{align*}Ef_1&=\underset{\xi\in \Lambda_{\lambda N_1}, \xi_1=1, |\xi_2|\leq \frac{1}{N_1}}{\sum} e^{2\pi i (\xi \cdot x+|\xi|^2 t)}\\
Ef_2&=\underset{\xi\in \Lambda_{\lambda N_1}, \xi_1=0, |\xi_2|\leq \frac{1}{N_1}}{\sum} e^{2\pi i (\xi \cdot x+|\xi|^2 t)}.
\end{align*}

$|Ef_1|$ is about $\lambda$ in the box of height $N_1^2$ (i.e. the $t$ direction), width $N_1$, (i.e the $x_{2}$ direction) and length $(\lambda N_1)^2$, (i.e. the $x_{1}$ direction) centered at origin. $|Ef_2|$ is the same size in the same box. Moreover, $Ef_1$ and $Ef_2$ both have periodicity around $\lambda N_1$ in $x_2$. 
\begin{align*}
\|Ef_1Ef_2\|_{L^2(w_{\Omega})}^2 &\gtrsim \lambda^4 N_1^2 \cdot N_1 \cdot (\lambda N_1)^2  \cdot \lambda N_1\\ 
&\gtrsim \lambda^7 N_1^6
\end{align*}

As calculated previously, $\|Ef_{j,\theta_j}\|_{L^4(w_{\Omega})}^4 =|\Omega|$. 
\begin{align*}
\Pi_{j=1}^2 (\sum_{|\theta_j|=\frac{1}{\lambda N_1}}\|Ef_{j, \theta_j}\|_{L^{4}(w_{\Omega})}^2) &\lesssim \lambda^2\cdot |\Omega| \\
&\lesssim \lambda^6 N_1^6.
\end{align*}
This example shows that when $d=2$, the term with $\frac{1}{\lambda}$ is sharp. 

When $d\geq 3$, we consider the example when 
\begin{align*}Ef_1&=\underset{\xi\in \Lambda_{\lambda N_1}, \xi_1=1, |(\xi_2,\dots, \xi_d)|\leq \frac{N_2}{N_1}}{\sum} e^{2\pi i (\xi \cdot x+|\xi|^2 t)}\\
Ef_2&=\underset{\xi\in \Lambda_{\lambda N_1}, \xi_1=0, |(\xi_2,\dots, \xi_d)|\leq \frac{N_2}{N_1}}{\sum} e^{2\pi i (\xi \cdot x+|\xi|^2 t)}.
\end{align*}

Notice that we construct the example in $d\geq 3$ differently, the support of $f_j$ is in a thin plate of radius $\frac{N_2}{N_1}$ instead of the $\frac{1}{N_1}$ as in $2$--dimensional example. 

$|Ef_1|$ is about $(\lambda N_2)^{d-1}$ in a box  of size $(\frac{N_1}{N_2}) \times\cdots \times \frac{N_1}{N_2} \times (\frac{N_1}{N_2})^2 \times (\lambda N_1)^2$.  $|Ef_2|$ is about $(\lambda N_2)^{d-1}$ in the same  box. Both $Ef_1$ and $Ef_2$ has periodicity around $\lambda N_1$ in $x_2, \dots ,x_d$--directions. 

\begin{align*}
\|Ef_1Ef_2\|_{L^2(w_{\Omega})}^2 &\gtrsim (\lambda N_1)^{4(d-1)} (\frac{N_1}{N_2} )^{d+1} (\lambda N_1)^2 (\lambda N_1)^{d-1} \\
&\gtrsim \lambda^{5d-3} N_1^{2d+2} N_2^{3d-5} 
\end{align*}
\begin{align*}
\Pi_{j=1}^2 (\sum_{|\theta_j|=\frac{1}{\lambda N_1}}\|Ef_{j, \theta_j}\|_{L^{4}(w_{\Omega})}^2) &\lesssim (\lambda N_2)^{2(d-1)}\cdot |\Omega| \\
&\lesssim \lambda^{4d-2} N_1^{2d+2} N_2^{2d-2}
\end{align*}
This example shows that when $d\geq 3$, the term with $\frac{N_2^{d-3}}{\lambda}$ is sharp.

\bibliographystyle{abbrv}
\bibliography{BG}
\end{document}